\documentclass[a4paper,11pt]{amsart}
\usepackage{amsmath,amsfonts,amssymb,xcolor}

\makeatletter
\def\@settitle{\begin{center}%
  \normalfont\Large\bfseries
  \@title % Use \@title to keep it dynamic
  \end{center}%
}
\makeatother

\newtheorem{theorem}{Theorem}[section]

\newtheorem{corollary}[theorem]{Corollary}
\newtheorem{lemma}[theorem]{Lemma}

\newtheorem{remark}[theorem]{Remark}

\def\diag{{\rm diag}\,}

\def\Span{\mathop{\rm Span}\,}

\def\Re{{\rm Re}}

\def\IC{{\mathbb{C}}}
\def\IR{{\mathbb{R}}}

\def\RR{\mathbb{R}}

\def\KK{{\mathbb{K}}}

\def\HH{{\mathbb{H}}}

\def\FF{{\mathbb F}}
\def\CC{{\mathbb C}}
\def\Span{\mathop{\mathrm{Span}}}
\def\sgno{\mathop{\mathrm{sgn}_0}}

\def\A{{\mathcal A}}
\def\IF{{\mathbb F}}

\def\RR{{\mathbb R}}
\begin{document}
\openup 1\jot

\title[Classification of abelian finite-dimensional $C^*$-algebras]{Classification of abelian finite-dimensional $C^*$-algebras by orthogonality}

\author[B. Kuzma]{Bojan Kuzma}
\address{University of Primorska, Glagolja\v{s}ka 8, SI-6000 Koper, Slovenia, and
Institute of Mathematics, Physics, and Mechanics, Jadranska 19, SI-1000 Ljubljana, Slovenia.}
 \email{bojan.kuzma@upr.si}

\author[S. Singla]{Sushil Singla}
\address{Department of Mathematics, Indian Institute of Science, Bengaluru 560012, India, and University of Primorska, Glagolja\v{s}ka 8, SI-6000 Koper, Slovenia,}
 \email{ss774@snu.edu.in}


\begin{abstract}
The main goal of the article is to prove  that if $\mathcal A_1$ and $\mathcal A_2$ are Birkhoff-James isomorphic  $C^*$-algebras over the fields $\mathbb F_1$ and $\mathbb F_2$, respectively and if $\mathcal A_1$ finite-dimensional, abelian of  dimension greater than one, then $\mathbb F_1=\mathbb F_2$ and $\mathcal A_1$ and $\mathcal A_2$ are (isometrically) $\ast$-isomorphic $C^*$-algebras. Furthermore, it is also proved that for a finite-dimensional $C^*$-algebra $\A$, we have $\mathcal L_{\A}^\bot$ is the sum of  minimal ideals which are not skew-fields and $\mathcal L_{\A}^{\bot\bot}$ is the  sum of minimal ideals which are skew-fields, where $\mathcal L_\A$ denotes the set of all left-symmetric elements in $\mathcal A$ and for any subset $\mathcal S\subseteq \mathcal A$, the set $\mathcal S^\bot$ represents the set of all elements of $\mathcal A$ which are Birkhoff-James orthogonal to $\mathcal S$. A procedure to extract the   minimal ideals which are (commutative) fields is also given. 
 \end{abstract}

\maketitle

\section{Introduction}

By Gelfand transform every unital abelian complex $C^*$-algebra $\mathcal A$ is $\ast$-isomorphic to $C(X)$, the space of complex-valued continuous functions on some compact Hausdorff space. This translates the study of algebraic properties 
 to the study of topological properties (and vice-versa):
 $C(X)$ is $\ast$-isomorphic to $C(Y)$ if and only if $X$ and $Y$ are homeomorphic topological spaces.  
Formally, Gelfand transform is a contravariant equivalence between the category of unital abelian $C^*$-algebras and the category of the space of compact Hausdorff spaces (see, e.g., \cite{blackadar, conway, goodearl} for more information). 

Recently Tanaka \cite[Theorem 3.5, Corollary 3.6]{tanaka3} showed that the same can be achieved by studying the geometrical properties rather than the topological ones: He characterized abelian complex $C^*$-algebras among all  complex $C^*$-algebras 
 by using only  the underlying geometric structure. Moreover, he showed that two complex abelian $C^*$-algebras are $\ast$-isomorphic if and only if their geometric structures are homeomorphic.

\iffalse
two unital abelian $C^*$-algebras are isomorphic if and only if their spectrum (i.e., the collection of non-zero multiplicative linear functionals, equipped with weak$^*$-topology) are homeomorphic as topological spaces. Thus, all the information about abelian $C^*$-algebra $\mathcal A$ is encoded completely in  the space of complex-valued continuous functions on the specturm $\Sigma_{\mathcal A}$, denoted by $C(\Sigma_{\mathcal A})$. Gelfand transform $\mathcal G$ is a map from a unital abelian $C^*$-algebra $\mathcal A$ to $C(\Sigma_{\mathcal A})$, where $\Sigma_{\mathcal A}=\{f\colon\mathcal A\rightarrow\mathbb C;\;\; f\text{ is a non-zero homomorphism}\}$ is maximal ideal space of $\mathcal A$, and defined as $\mathcal G(a)(f)=f(a)$ for all $f\in \Sigma_{\mathcal A}$ and $a\in\mathcal A$. It is a fact that Gelfand transform is a continuous homomorphism of norm one 
  for any unital $C^*$-algebra $\mathcal A$ and it turns out to be onto isometric $\ast$-isomorphism in case $\mathcal A$ is an abelian $C^*$-algebra.
\fi 

The geometric structure   was initially defined in terms of \textit{Birkhoff-James orthogonality} (see~\cite[Definition 3.4]{tanaka1}). 
 Figuratively speaking, 
 suppose we obtain a cast of the closed unit ball of the $C^\ast$-algebra norm. We are allowed to examine it with a sufficiently long, infinitesimally thin needle by placing it tangentially in various directions onto the unit sphere of the norm, then translating it parallelly   to the center of the ball and  examining the  points which the translated needle cuts out from the boundary. If $x$ is the touching point of the needle and $y$ is the cut-out point of the translated needle, then the tangentiality of the needle at point $x$ in direction $y$ is equivalent to $\|x+\lambda y\|\ge \|x\|$ for each scalar $\lambda$, that is, to Birkhoff-James orthogonality of $x$ and $y$ (for complex $C^\ast$-algebras, the probing needle has two real dimensions). The geometric structure was defined in terms of maximal faces of  $C^*$-norm's unit ball and requires no knowledge of algebraic operations. 
 Following the approach outlined above we, in \cite[Theorem 2.1]{simple}, 
 completely classified  
 the objects in the categories of  real or complex
  finite-dimensional simple $C^*$-algebras by using only the relation of Birkhoff-James orthogonality, that is, by relying only on the norm structure alone. Another important aspect of our work was that we worked with real as well as complex $C^*$-algebras simultaneously and even gave a procedure to characterize the underlying field when the dimension of space is greater than one. The theory of real $C^*$-algebras is similar to complex $C^*$-algebras, see \cite{Con01, goodearl, Li03, Sch93}  (and~\cite{Rosenberg} for a review of their applications) though  being able to characterize the underlying field of a  given  real or complex $C^*$-algebra, from Birkhoff-James orthogonality alone, was still  a bit surprising.

In this article we continue our work in the categories of finite-dimensional real or complex $C^*$-algebras and characterize \textit{(pseudo-)abelian $C^*$-algebras} together with the underlying fields when the dimension of $C^*$-algebra is greater than one. A few notations are in order. In the sequel, $\mathcal A$ will stand for a finite-dimensional $C^*$-algebra over the field $\mathbb F\in\{\mathbb R, \mathbb C\}$. We will denote the matrix block decomposition of a complex $C^*$-algebra $\mathcal A$ by
\begin{equation}\label{eq:C*(complex)}
  \mathcal M_{n_1}(\mathbb C)\oplus\dots\oplus\mathcal M_{n_\ell}(\mathbb C)
\end{equation}
 for some positive integers $n_1\leq \dots\leq n_\ell$. Similarly, the matrix block decomposition of a real $C^*$-algebra $\mathcal A$ will be denoted by
 \begin{equation}\label{eq:C*(real)}
\mathcal M_{n_1}(\mathbb K_1)\oplus\dots\oplus\mathcal M_{n_\ell}(\mathbb K_\ell)
\end{equation}
 where $\mathbb K_i\in\{\mathbb R, \mathbb C, \mathbb H\}$ with real dimension $d_1,\dots, d_\ell$ such that $d_i\leq\dots\leq d_j$ whenever $n_i=\dots= n_j$ (every~$\A$ has such decomposition, see \cite[Theorem 1.5, Theorem 8.4]{goodearl} for more information). 

Recall that a non-zero two-sided ideal of $\A$ is minimal if it does not properly contain any other non-trivial two-sided ideal. We will refer to  the sum of those minimal ideals of $\A$ which are skew-fields as \textit{a pseudo-abelian summand} of $\A$ and we will refer to the sum of abelian minimal ideals  as an \textit{abelian summand} of $\A$.  If $C^*$-algebra is decomposed as in \eqref{eq:C*(complex)} or \eqref{eq:C*(real)}, then its minimal ideals coincide with  the individual blocks, its pseudo-abelian summand coincides with the sum of all blocks of size one, and its abelian summand coincides with the sum of all blocks of size one over the real or complex field; see, e.g., Wedderburn-Artin theorem~\cite[V.4.6]{grove}. Therefore, if $\mathbb F=\mathbb C$ the pseudo-abelian summand coincides  with the abelian summand of $\A$. Also, if $\A$ is an abelian $C^*$-algebra, then its pseudo-abelian summand equals $\A$. However, in case $\FF=\RR$ the pseudo-abelian summand might contain a quaternionic block in which case it is not abelian. We will say that   $\mathcal A$ is a  pseudo-abelian $C^*$-algebra if it equals to its pseudo-abelian summand. In the case when $\FF=\CC$ this is the same as an abelian $C^*$-algebra. 

Let us  briefly discuss also the definition and basic properties of Birkhoff-James (BJ) orthogonality. For two vectors $v, w$ in a normed space $V$ over a field $\mathbb F$, $v$ is said to be BJ orthogonal to $w$, denoted by $v\perp w$ if $$\|v\|\leq \|v+\lambda w\|\text{ for all }\lambda \in \FF.$$  One can easily see that this relation is homogeneous and that, equivalently, $v\perp w$ if and only if $f_v(w)=0$ for some  supporting functional  $f_v$  at $v$ (that is, $\|f_v\|=1$ and $f_v(v)=\|v\|$), see \cite[Theorem 2.1]{james1} or \cite[Equation 2.1]{guterman}. Therefore, if we define the \textit{outgoing neighborhood} of $v$ by
\begin{equation}\label{eq:perp}
  v^\bot=\{w\in V;\;\; v\perp w\},
\end{equation} then we have \begin{equation}\label{outgoingnbd}v^\bot =
\bigcup\{\mathop{\mathrm{ker}}f;\;\; f\text{ is supporting functional at } v\}.
 \end{equation} A bijective map $\phi\colon V\rightarrow V'$ is  a \textit{BJ isomorphism} between $V$ and $V'$ if $$v\perp w\iff \phi(v)\perp \phi(w)\quad\text{ for all } v,w\in V.$$ Two normed spaces $V$ and $V'$ are  \textit{BJ isomorphic} if there exists a BJ isomorphism between them.

We can now state our main results. Recall that a finite-dimensional $C^*$-algebra is  pseudo-abelian  if it   contains contains only blocks of size one in its matrix block decomposition.
 
\begin{theorem}\label{abelian} Let $\A_1$ and $\A_2$ be two BJ isomorphic $C^*$-algebras over the fields $\FF_1$ and $\FF_2$. If $\A_1$ is finite-dimensional pseudo-abelian $C^*$-algebra with $\dim\A_1\geq 2$, then the following are true:
\begin{enumerate}
    \item  $\FF_1=\FF_2$,
    \item $\A_1$ and $\A_2$ are isomorphic as $C^\ast$-algebras, so in particular, $\A_2$ is pseudo-abelian and $\dim\A_1=\dim\A_2$.
\end{enumerate}
\end{theorem}
It is immediate that if, in the Theorem~\ref{abelian}, $\mathcal A_1$ is a finite-dimensional abelian $C^*$-algebra, then $\mathcal A_2$ is also abelian. We further remark that  BJ orthogonality alone cannot  determine the underlying field in  one-dimensional $C^*$-algebras because the real $C^*$-algebra $\mathcal M_1(\mathbb R)=\RR$ and the complex $C^*$-algebra $\mathcal M_1(\mathbb C)=\CC$ are BJ isomorphic; see~\cite[Example 2.2]{simple}.

The above theorem can be seen as a partial extension of  a recent result  \cite[Corollary 3.6]{tanaka3} which Tanaka proved for complex $C^*$-algebras: if two complex $C^*$-algebras $\A_1$ and $\A_2$ are BJ isomorphic and one of them is abelian, then they are isomorphic as $C^\ast$-algebras. Within Theorem~\ref{leftsymmetric} below we will  further generalize Theorem~\ref{abelian} to include also the possibility when $\A_1,\A_2$ are  BJ isomorphic but not pseudo-abelian; when combined, the two theorems  imply that the pseudo-abelian summands of  $
\A_1$ and of $\A_2$ are isomorphic as $C^\ast$-algebras  (provided $\A_1,\A_2$ are not one-dimensional). The characterization is based on the notion of left-symmetricity (see~\cite{Komuro-Saito-Tanaka, Sain-Ghosh-Paul} and also~\cite{rightsymmetricturnsek}). A vector $v$ in a normed space $V$ is  \textit{left-symmetric}  if $$(v\perp w)\implies(w\perp v),$$ and is  \textit{right-symmetric} if $(w\perp v)\implies (v\perp w)$. By using the outgoing neigborhood defined within~\eqref{eq:perp} and the \textit{incoming neighborhood} ${}^\bot v:=\{w\in V;\;\;w\perp v\}$ it is easily seen that $v$ is left-symmetric if and only if
$$v^\bot\subseteq{}^\bot v,$$ and is right-symmetric if and only if the reversed inclusion holds. For a subset $\mathcal S$ of $V$, we will use the notation $\mathcal L_{\mathcal S}$ for the set of all left-symmetric vectors relative to $\mathcal S$, i.e.
$$\mathcal L_{\mathcal S}:=\{v\in\mathcal S;\;\;v^\bot	\cap {\mathcal S}\subseteq {}^\bot v\cap {\mathcal S}\}.$$
In particular, if $\mathcal S=V$, then $\mathcal L_{V}$ is the set of all left-symmetric vectors. We will also use the notations $\mathcal L_{\mathcal S}^\bot:=\bigcap_{v\in {\mathcal L}_{\mathcal S}}v^\bot$ and $\mathcal L_{\mathcal S}^{\bot\bot}:=\bigcap_{v\in {\mathcal L}_{\mathcal S}^\bot} v^\bot$.

Given a finite-dimensional $C^*$-algebra $\mathcal A$, we will call  the sum of all its minimal ideals which are not skew-fields (that is, the sum of all its blocks of sizes bigger than one) to be \textit{the nonpseudo-abelian summand} of $\A$. In particular, a finite-dimensioal $C^*$-algebra is a sum of its nonpseudo-abelian and its pseudo-abelian summands.
\begin{theorem}\label{leftsymmetric} Let $\mathcal A$ be a finite-dimensional $C^*$-algebra
 over field $\mathbb F$. Then, $\mathcal L_{\A}^\bot$ is the nonpseudo-abelian summand  and $\mathcal L_{\A}^{\bot\bot}$ is the pseudo-abelian summand of $\A$.
\end{theorem}

 The set   $\mathcal L_\A^{\bot\bot}$  is therefore
 a $C^*$-algebra and in fact classifies pseudo-abelian finite-dimensional $C^*$-algebras as follows:

\begin{corollary}\label{abeliancharacterization}
   Let $\A$ be a finite-dimensional $C^*$-algebras over $\FF$. Then $\A$ is a pseudo-abelian $C^*$-algebra if and only if $\mathcal L_\A^{\bot\bot}=\A$. 
\end{corollary}

Theorem~\ref{abelian} suggests that BJ orthogonality alone can determine whether a finite-dimensional $C^*$-algebra is abelian and, if not, to extract its abelian summand. We will provide a positive solution to this problem in the Section \ref{section4}, once we  develop the necessary machinery to formalize the procedure that isolates the quaternionic blocks in the pseudo-abelian summand.
 
\begin{remark}
(a) Finite-dimensional complex $C^*$-algebras are von-Neumann algebras (see, e.g., \cite{blackadar} for more on von-Neumann algebras).  Recall that an element $p$ of von-Neumann algebra $\A$ is called a central projection if $p^2=p=p^\ast$ (a projection) and $p$ commutes with all other elements of $\A$. A projection $p\in\A$ is called abelian  if $p\A p$ is commutative. In case of factor von-Neumann algebras $\mathcal M_n(\mathbb C)$, central abelian projections exist only when $n=1$. Thus, the abelian summand equals the complex linear span of abelian central projections. A non-zero projection $p$ is minimal if the only
non-zero projection $q \in\A$  such that $q \le  p$ is $q = p$. Since
$p\A p$ is also a von Neumann algebra with $p$ as an identity, it is easily seen that this is equivalent to the fact that  $p\A p$ is a field. Thus, in case $\FF=\CC$,  Theorem~\ref{leftsymmetric} says that $\mathcal L_{\A}^\bot$ is the complex linear span of minimal non-central projections and $\mathcal L_{\A}^{\bot\bot}$ is the complex linear span of abelian central projections (or abelian summand).

(b) In a related study~\cite[Theorem 3.2]{Komuro-Saito-Tanaka} the authors classified elements, left-symmetric relative to the positive cone of general complex $C^*$ algebras. These  are exactly scalar multiples of minimal projections.
\end{remark}

A proof of Theorems \ref{abelian} and  \ref{leftsymmetric} will be given in Section \ref{sect:prfs}. In Section~\ref{section2} we characterize left-symmetric elements and right-symmetric elements in Lemmas \ref{left} and \ref{right}. In Lemma \ref{right} we prove that right-symmetric elements are exactly scalar multiples of unitaries. This extends \cite[Theorem 2.5]{rightsymmetricturnsek}  to general finite-dimensional $C^*$-algebras. As a consequence, every BJ isomorphism will map the set of  scalar multiples of unitaries to itself. 
Section \ref{section3} is devoted to developing the tools to  prove Theorem \ref{abelian}. In \eqref{dimvslefteq} and \eqref{dimformula}, formulas to find the dimension of an pseudo-abelian $C^*$-algebra $\A$ are provided and in Corollary \ref{lem:Rn-vs-Cn}, a characterization of the underlying field of $\A$ is given, provided the dimension of $\A$ is greater than one. Lemma \ref{lem:totalnrblocks} gives a procedure to find the number of blocks in the matrix block decomposition of an pseudo-abelian $C^*$-algebra. In section \ref{section4} we extract the abelian summand and give a characterization  of abelian $C^*$-algebras in terms of BJ orthogonality.

\section{Symmetricity and smoothness}\label{section2}

Any $A=\bigoplus_{k=1}^\ell A_k\in\bigoplus_{k=1}^\ell\mathcal M_{n_k}(\mathbb K_k)$ acts on a column  vector  $x=\bigoplus_{k=1}^\ell x_k\in\widehat{\mathbb K}=\bigoplus_{k=1}^\ell \mathbb K_k^{n_k}$ by $Ax= \bigoplus_{k=1}^\ell (A_kx_k)$. We let (a row vector) $x^\ast$ be its conjugate transpose. If  needed we will  regard $\KK_k^{n_k}$ as a right-vector space over the  (skew) field $\KK_k$;  the above action, when restricted to a summand~$\KK_k^{n_k}$, then induces a $\KK_k$-linear operator. We regard $\widehat{\mathbb K}$ as a (right) $\FF$-vector space and equip it with a natural $\FF$-valued inner product by $$\Big\langle \bigoplus_{k=1}^\ell x_k, \bigoplus_{k=1}^\ell y_k\Big\rangle_{\FF}=\sum\limits_{k=1}^\ell \langle x_k, y_k\rangle_{\mathbb F}\quad \hbox{where} \quad \langle x_k, y_k\rangle_{\mathbb F}:=\begin{cases}
                                                                                  \mathrm{Re}\,y_k^\ast x_k, & \mbox{if } {\mathbb F}=\RR \\
                                                                                  y_k^\ast x_k, & \mbox{if } {\mathbb F}=\CC
                                                                                \end{cases}.$$
Notice that  $\langle x_k,y_k\rangle_{\FF}=y_k^\ast x_k$ only in complex $C^\ast$-algebras in which case $\KK_k=\CC$ for each $k$ and $\widehat{\mathbb K}=\CC^{n_1+\dots+n_\ell}$. Notice also that
\begin{equation}\label{eq:inner-product-equality}
  \langle x_{k}\lambda,y_k\lambda\rangle_{\FF}=|\lambda|^2\langle x_k,y_k\rangle_{\FF};\qquad \lambda \in\KK_k
\end{equation} which is clear if $\FF=\CC$ and is also clear if $\FF=
\RR$ and $x,y$ belong to $\RR^n$ or $\CC^n$. However, if $x_k,y_k\in\HH^n$ we have $\langle x_k\lambda,y_k\lambda\rangle_{\FF}=\mathrm{Re}\,\overline{\lambda} (y_k^\ast x_k)\lambda$. Here we decompose $y_k^\ast x_k\in\KK_k=\HH$ into its real and purely imaginary part and use that the conjugation $\overline{\lambda} (y_k^\ast x_k)\lambda$ maps purely imaginary part again into purely imaginary part, so  $\mathrm{Re}\,\overline{\lambda} (y_k^\ast x_k)\lambda=\overline{\lambda} \mathrm{Re}\,(y_k^\ast x_k)\lambda=|\lambda|^2\mathrm{Re}\,(y_k^\ast x_k)$.

This inner product defines a norm on $\widehat{\mathbb K}$ and the induced operator norm for $A=\bigoplus_{k=1}^\ell A_k\in\bigoplus_{k=1}^\ell\mathcal M_{n_k}(\mathbb K_k)$ coincides with $C^\ast$-norm and satisfies
 $$\big\|\bigoplus_{k=1}^\ell A_k\big\|=\max\{\|A_k\|;\;\;1\leq k\leq \ell\}.$$ For its computation, we recall the singular value decomposition for $\mathbb {\mathcal M}_n(\mathbb K)$. It states that for $A\in\mathbb {\mathcal M}_n(\mathbb K)$, there exists $\KK$-orthonormal basis $\{x_1, \dots, x_n\}$ and $\{y_1, \dots, y_n\}$  (i.e., $x_i^\ast x_j=y_i^\ast y_j=\delta_{ij}$, Kronecker delta) such that $A=\sum\limits_{i=1}^n\sigma_iy_ix_i^*$, where $\sigma_1\geq\dots\geq \sigma_n\ge0$ are singular values of $A$ arranged in the decreasing order, with $\|A\|=\sigma_1$ (see \cite[Theorem 7.2]{zhang} for the singular value decomposition for $\mathbb {\mathcal M}_n(\mathbb H)$).

For $A=\bigoplus_{k=1}^\ell A_k\in\A$, we define $$M_0(A) =\{x \in \widehat{\mathbb K} ;\;\; \|Ax\|=\|A\|\|x\|\}.$$  We also define $M_0^*(A)=\bigoplus_{k=1}^\ell M_0^*(A_k)$ where $$M_0^*( A_k)=\begin{cases}
    M_0(A_k), & \mbox{ if }\|A_k\|= \|A\| \\
    0_{k}, & \mbox{ if } \|A_k\|<\|A\|
\end{cases}.$$
Notice that $M_0(A_k)$ is a $\mathbb K_k$-subspace of $\mathbb K_k^{n_k}$ (see, e.g.,~\cite[Lemma 3.1]{simple}), hence in particular an $\mathbb F$-subspace, so $M_0^\ast(A)$ is also an $\mathbb F$-subspace of $\widehat{\mathbb K}$. We will prove in the next lemma that $M_0(A)=M_0^\ast(A)$; therefore $M_0(A)$ is also a $\mathbb F$-subspace of $\widehat{\mathbb K}$. We will use $0_{n_k}$ for the zero matrix in $\mathcal M_{n_k}(\KK_k)$.

\begin{lemma}\label{M_0(A)}  Let $\A=\bigoplus_{k=1}^\ell\mathcal M_{n_k}(\mathbb K_k)$ be a $C^*$-algebra and let $A,B\in\A$. Then:
\begin{enumerate}
    \item[(i)]    $B\in A^\bot$ if and only if there exists a normalized vector $x\in M_0(A)$ such that $\langle Ax,Bx\rangle_{\FF} =0$.
    \item[(ii)] If $A =\bigoplus_{k=1}^\ell A_k$ is its decomposition, then $$M_0^*( A_k)\subseteq M_0(A)=M_0^*(A) \subseteq \bigoplus_{k=1}^\ell M_0(A_k).$$
\end{enumerate}
Consequently, if  $\|A_i\|<\max\{\|A_k\|;\;\;1\leq k\leq\ell\}=\|A\|$, then $$A^\bot = \big(A_1\oplus\dots\oplus A_{i-1}\oplus 0_{n_i}\oplus A_{i+1}\oplus\dots\oplus A_\ell\big)^\bot.$$
\end{lemma}
\begin{proof} For (i) we note that the BJ orthogonality of two vectors $v$ and $w$ in a normed space $V$ depends only on the two-dimensional subspace generated by $v$ and $w$. With this in mind,  if $\FF=\RR$,  consider the embedding of  real $C^*$-algebras $\mathcal M_n(\mathbb C)$ and $\mathcal M_n(\mathbb H)$ into $\mathcal M_{2n}(\mathbb R)$ and $\mathcal M_{4n}(\mathbb R)$, respectively. This way,  a $C^\ast$-algebra $\A=\bigoplus_{k=1}^\ell\mathcal M_{n_k}(\mathbb K_k)$ embeds into $\mathcal M_{\sum d_kn_k}(\mathbb F)$ where $d_k=\dim_{\IR}{\mathbb K_k}$ in case of $\mathbb F=\mathbb R$ and where  $d_k=\dim_{\IC}\KK_k=1$ in case of  $\mathbb F=\mathbb C$  (since $\mathbb K_k=\mathbb C$ when $\IF=\IC$). Then,~(i) follows by Stampfli-Magajna-Bhatia-\v Semrl classification (see \cite[Theorem 1]{bhatia} or, for some historical background, \cite[Proposition 3.2]{simple} and \cite[page 2716]{grover1}) applied on $\mathcal M_{\sum d_kn_k}(\mathbb F)$.

For (ii), we first show $M_0(A)=M_0^\ast(A)$. Let $x=\bigoplus_{k=1}^\ell x_k \in M_0(A)$. Then, $\|Ax\|=\|A\|\|x\|$ and so 
\begin{align*}\|A\|^2\|x\|^2=\|Ax\|^2=\sum\limits_{k=1}^\ell\|A_kx_k\|^2\leq \sum\limits_{k=1}^\ell\|A_k\|^2\|x_k\|^2\leq \max\limits_{1\leq k\leq\ell}\|A_k\|^2\big(\sum\limits_{k=1}^\ell\|x_k\|^2\big)=\|A\|^2\|x\|^2.\end{align*} Now, it implies equalities overall so that $$\|A_kx_k\|^2=\|A_k\|^2\|x_k\|^2=(\max\limits_{1\leq k\leq\ell}\|A_k\|^2)\|x_k\|^2=\|A\|^2\|x_k\|^2\text{ for all }1\leq k\leq \ell.$$ Therefore, if $\|A_i\|<\max\limits_{1\leq k\leq\ell}\|A_k\|$ we have $x_i=0$, while if $\|A_i\|=\max\limits_{1\leq k\leq\ell}\|A_k\|$, we have $x_i\in M_0(A_i)$. It implies $M_0(A)\subseteq M_0^\ast(A)$. Conversely, if $x=\bigoplus_{k=1}^\ell x_k \in M_0^*(A)$, let $\Lambda$ be the collection of all indices $i$ such that $\|A_i\|=\|A\|$. By definition of $M_0^\ast(A)$ we have $x_k=0$ if $k\notin\Lambda$, so that
$$\|Ax\|^2=\sum\limits_{k=1}^\ell\|A_kx_k\|^2=\sum\limits_{i\in \Lambda}\|A_ix_i\|^2 = \|A\|^2\sum\limits_{i\in \Lambda}\|x_i\|^2=\|A\|^2\|x\|^2.$$ This proves $x\in M_0(A)$, hence $M_0^\ast(A)\subseteq M_0(A)$. The other containment in (ii) follows directly from the definitions, while the last statement of the lemma follows from (i) and (ii).
\end{proof}

We  now  begin to   investigate algebraic properties of elements of $\A$ with BJ orthogonality.  Let us record a trivial observation which classifies the $0$ element in terms of BJ orthogonality.  We will tacitly used it in many subsequent lemmas when  claiming that BJ orthogonality alone determines a relevant property:
$$A=0 \Longleftrightarrow A\perp A.$$

\begin{lemma}\label{left} Let $\mathcal A=\mathcal M_{n_1}(\mathbb K_1)\oplus\dots\oplus\mathcal M_{n_\ell}(\mathbb K_\ell)$ with $n_1=\dots=n_p=1$ and $n_{p+1},\dots, n_\ell\ge2$ for some $p\geq 0$. If $p=0$, then $\mathcal A$ has no non-zero left-symmetric elements. If $p\geq 1$, then  the following are equivalent for a non-zero element $A\in \mathcal A$:
\begin{itemize}
    \item[(i)] $A$ is left-symmetric element of $\A$,
    \item[(ii)] The block decomposition of $A$ has only one non-zero entry and it belongs to  $\mathcal M_1(\mathbb K_i)$ for some $i\in [1,p]$.
\end{itemize}
\end{lemma}
\begin{proof} Let $A$ be a normalized left-symmetric and let $i$ be such that  $\|A_i\|=\|A\|=1$. Without loss of generality,  $A_i=\Sigma_i=\diag(\sigma_1^i,\dots, \sigma_{n_i}^i)$ is already in its singular value decomposition form where $1=\|A_i\|=\sigma_1^i\ge \dots\ge \sigma_{n_i}^i\ge0$. Consider now a matrix $B_i=e_2^i(e_2^i)^\ast\in\mathcal M_{n_i}(\mathbb K_i)$ and let 
 $B=\big(\bigoplus_{k=1}^{i-1}0_{n_k}\big)\oplus B_i\oplus \big(\bigoplus_{k=i+1}^\ell0_{n_k}\big)$.
Notice that $A$ attains its norm on 
 $x=\big(\bigoplus_{k=1}^{i-1}0_k\big)\oplus e_1^i\oplus \big(\bigoplus_{k=i+1}^\ell0_k\big)$
and $$\langle Ax, Bx\rangle_{\FF}= \langle A_ie_1^i,B_ie_1^i\rangle_{\mathbb F}=0,$$ so $A\perp B$ by Lemma \ref{M_0(A)}. Being left-symmetric, this implies $B\perp A$. Note also that $B$ attains its norm only on 
vectors $y\in\big(\bigoplus_{k=1}^{i-1}0_k\big)\oplus e_2^i\mathbb K_i\oplus \big(\bigoplus_{k=i+1}^\ell0_k\big)$,
and maps them into themselves, so $B\perp A$ forces  $$0=\langle By, Ay\rangle_{\FF}=\langle B_ie_2^i,A_ie_2^i\rangle_{\mathbb F}=\sigma_2^i.$$

Similarly, $\sigma_j^i=0$ for all $2\leq j\leq n_i$. Thus, $A_i=\diag(1,0,\dots,0)= e_1^i(e_1^i)^*$. If now $n_i>1$, consider $B_i=(a e_1^i+b e_2^i)(a e_1^i+ be_2^i)^\ast-\tfrac{1}{3}(-b e_1^i+ ae_2^i)(-b e_1^i+ ae_2^i)^\ast$ where $(a,b)=(-1/2,\sqrt{3}/2)$ and 
 $B=\big(\bigoplus_{k=1}^{i-1}0_{n_k}\big)\oplus B_i\oplus \big(\bigoplus_{k=i+1}^\ell0_{n_k}\big)$.
Notice that $A$ attains its norm on
 $x=\big(\bigoplus_{k=1}^{i-1}0_{k}\big)\oplus e_1^i\oplus \big(\bigoplus_{k=i+1}^\ell0_{k}\big)$
and that $$\langle Ax, Bx\rangle_{\FF}=\langle A_ie_1^i,B_ie_1^i\rangle_{\mathbb F}=0,$$
so $A\perp B$. Notice also that $B$ attains its norm only on a multiple of 
 $y=\big(\bigoplus_{k=1}^{i-1}0_{k}\big)\oplus(a e_1^i+ b e_2^i)\oplus \big(\bigoplus_{k=i+1}^\ell0_{k}\big)$
and that $\langle By, Ay\rangle_{\FF}= \langle B_iy_i,A_iy_i\rangle_{\FF}=a^2\neq0$, so $B\not\perp A$. Hence, $A$ is not left-symmetric, a contradiction.

The only possibilities left for $A$ are of form $\alpha_1\oplus\dots\oplus\alpha_p\oplus A_{p+1}\oplus\dots\oplus A_\ell$, with $\|A_k\|<\|A\|$ for $k> p$ (we identified the  $1$-by-$1$ summands with scalars $\alpha_k\in\KK_k$). We first claim that each of $A_k$ is zero for $k>p$. Namely, we clearly have $$0^p\oplus \mathcal M_{n_{p+1}}(\KK_{p+1})\oplus\dots\oplus \mathcal M_{n_\ell}(\KK_\ell)\subseteq \big(\alpha_1\oplus\dots\oplus\alpha_p\oplus A_{p+1}\oplus\dots\oplus A_\ell\big)^\bot,$$ 
(here, $0^p$ denotes $p$ repeated zeros).

Since $A$ is left-symmetric, it implies that 
 $X=\big(\bigoplus_{k=1}^{p+i-1} 0_{n_k}\big)\oplus A_{p+i}\oplus\big(\bigoplus_{k=p+i+1}^{\ell} 0_{n_k}\big)$  
satisfies $X\perp A$. Clearly if $X\neq0$ it achieves its norm only inside its unique non-zero block. Then, however, $X\perp A$ is equivalent to $A_{p+i}\perp A_{p+i}$, so that $A_{p+i}=0$ since the only matrix in $\mathcal M_{n_{p+i}}(\KK_{p+i})$ which is orthogonal to all matrices is a zero matrix.

Hence, the only possibilities left for $A$ are of form $\alpha_1\oplus\dots\oplus\alpha_p\oplus 0_{n_{p+1}}\oplus\dots\oplus 0_{n_\ell}$. We claim that only one of $\alpha_i$ can be non-zero. Without loss of generality $|\alpha_1|=\|A\|$. Assume  $\alpha_2\neq 0$ and notice that $$0\oplus\overline{\alpha_2}\oplus0^{p-2}\oplus 0_{n_{p+1}}\oplus\dots\oplus 0_{n_\ell}\in\big(\alpha_1\oplus\alpha_2\oplus\dots\oplus\alpha_p\oplus0_{n_{p+1}}\oplus\dots\oplus 0_{n_\ell}\big)^\bot,$$ but they are not mutually BJ orthogonal. Indeed, $\alpha_2=0$, and  the only  possibility for left-symmetric $A$ is that all, except possibly one, of  its blocks are zero and the non-zero block is of size $1$-by-$1$.

It is easily observed that these are indeed left-symmetric elements, namely:
given $A=\alpha\oplus\{0\}^{p-1}\oplus\big(\bigoplus_{k=p}^\ell 0_{n_k}\big)$, choose any $B=\beta\oplus\big(\bigoplus_{k=2}^\ell B_k)\in A^\bot$. Note that, modulo a  multiplication by scalars,  $x=1\oplus 0\oplus\dots\oplus 0$ is the only norm-attaining vector for $A$. Then, $0=\langle Ax, Bx\rangle_{\FF}= \langle \alpha\cdot 1, \beta\cdot 1\rangle_{\mathbb F} =\Re(\bar{\beta}\alpha)$ (or it equals $\bar{\beta}\alpha$ if $\FF=\CC$). Now, if $|\beta|=\|B\|$, then $x$ is also a norm-attaining vector for $B$ and we have $B\perp A$. If $|\beta|\neq \|B\|$, then it is clear that $B\perp A$. It implies $A$ is left-symmetric.\end{proof}

We  now characterize right-symmetric elements. The following lemma  which  generalizes \cite[Lemma 3.3]{simple} will be useful.

\begin{lemma}\label{lengthmaximal}
       Let $\A=\bigoplus_{k=1}^\ell\mathcal M_{n_k}(\mathbb K_k)$ be a $C^*$-algebra over $\FF$. For elements $A=\bigoplus_{k=1}^\ell A_k$ and $B=\bigoplus_{k=1}^\ell B_k$ we have $A^\bot\subseteq B^\bot$ if and only if for all $1\leq i\leq \ell$, we have $M_0^\ast(A_i)\subseteq M_0^\ast(B_i)$ and there exist non-zero $\alpha_i\in \mathbb F$ having the same modulus such that $A_ix_i=\alpha_i(B_ix_i)$ for all $x_i\in M_0^\ast(A_i).$
\end{lemma}
\begin{proof} Let $A^\bot\subseteq B^\bot$. If $\|A_i\|<\|A\|$, then $M_0^*(A_i)=\{0_{i}\}\subseteq M_0^*(B_i)$. Assume $\|A_i\|=\|A\|$, let $x_i\in M_0(A_i)$ be a normalized vector and consider the hyperplane $H=\{X=\bigoplus_{k=1}^\ell X_k\in\A;\;\; \langle X_ix_i,A_ix_i\rangle_{\mathbb F}=0\}$  contained in $A^\bot\subseteq B^\bot$. By \cite[Theorem 2.1]{james1}
(whose proof works over real as well as complex normed spaces), we have $|F(B)|=\|F\|\|B\|$, where $F(X)= \langle A_ix_i, X_ix_i\rangle_{\mathbb F}$ is an $\mathbb F$-linear functional on $\A$. Hence  $$|\langle A_ix_i, B_ix_i\rangle_{\mathbb F}|=|F(B)|=\|F\|\|B\| =\|A_ix_i\|\cdot\|x_i\|\cdot\|B\|,$$ (where last equality follows as an application of Cauchy-Schwarz inequality and the fact that $\|X_i\|\le \|X\|$). Now, we have 
\begin{align*}\|A_ix_i\|\cdot\|x_i\|\cdot\|B\|= |\langle A_ix_i, B_ix_i\rangle_{\mathbb F}|\leq \|A_ix_i\|\|B_ix_i\|\leq \|A_ix_i\|\cdot\|x_i\|\cdot\|B_i\|\leq \|A_ix_i\|\cdot\|x_i\|\cdot\|B\|.\end{align*}
So, we get equality throughout, so
$$\|B\|=\|B_i\|$$ and $\|B_ix_i\| = \|B_i\|\|x_i\|$. Also, by the condition of equality in Cauchy-Schwarz inequality, we get $A_ix_i=\lambda_i (B_ix_i)$ for some $\lambda_i=\lambda_i(x_i)\in\mathbb F$. Now, we prove $\lambda_i(x_i)$ is independent of $x_i\in M_0(A_i)$. First note that $|\lambda_i(x_i)|=\|A_i\|/\|B_i\|=\|A\|/\|B\|$.  Let $x_i, y_i\in M_0(A_i)$ be normalized vectors such that $A_ix_i=\mu_i (B_ix_i)$ and $A_iy_i=\nu_i (B_iy_i)$. We first observe that there exists a normalized vector $w_i\in M_0(A_i)$ such that $$\dfrac{\mu_i+\nu_i}{2}=\dfrac{\langle B_ix_i, A_ix_i\rangle_{\mathbb F}+\langle B_iy_i, A_iy_i\rangle_{\mathbb F}}{2}=\langle B_iw_i, A_iw_i\rangle_{\mathbb F};$$ in the last step we have used the convexity of $\{\langle B_ix_i, A_ix_i\rangle_{\mathbb F};\;\; \|x_i\|=1, x_i\in M_0(A_i)\}$, the numerical range of the compression of  $A_i^*B_i$ to the subspace $M_0(A_i)$ (recall that $A_i^*B_i$ is embedded into a suitable ${\mathcal M}_{n}(\RR), {\mathcal M}_{2n}(\RR)$ or ${\mathcal M}_{4n}(\RR)$). Again, $w_i\in M_0(A_i)$ implies that $A_iw_i = \tau_i(B_iw_i)$ for some~$\tau_i$ with $|\tau_i|=\|A\|/\|B\|$. Thus, $|\mu_i+\nu_i|=2\|A\|/\|B\|$. But $|\mu_i|=|\nu_i|=\|A\|/\|B\|$ and hence  $\mu_i=\nu_i$. So, $\lambda_i(x_i)$ is independent of $x_i$.

Also, $|\alpha_i|=\|A\|/\|B\|$ for all $i$, when $M_0^*(A_i)\neq \{0_{i}\}$ and in case $M_0^*(A_i)= \{0_{i}\}$, we can choose any $\alpha_i$, so in particular $\alpha_i=\|A\|/\|B\|$. The converse follows from Lemma \ref{M_0(A)} (i).\end{proof}

As an application of Lemma \ref{M_0(A)}, we also get a classification of right-symmetric elements. It  turns out that they are nothing but scalar multiples of unitary elements. For simple complex finite-dimensional $C^*$-algebra, this follows from \cite[Theorem 2.5]{rightsymmetricturnsek} (but see also \cite[Lemma 3.7]{simple}).

\begin{lemma}\label{right} The following are equivalent in a finite-dimensional $C^*$-algebra $\mathcal A$ over $\FF$:
\begin{itemize}
    \item[(i)] $A$ is a scalar multiple of some unitary $U$. 
    \item[(ii)] There does not exist a non-zero $B\in\A$ such that $A^\bot\subsetneq B^\bot$.
    \item[(iii)] $A$ is right-symmetric in $\mathcal A$.
\end{itemize}
\end{lemma}
\begin{proof} Without loss of generality, let $\mathcal A=\mathcal M_{n_1}(\mathbb K_1)\oplus\dots\oplus\mathcal M_{n_\ell}(\mathbb K_\ell)$. 

 (i) $\Longleftrightarrow$ (ii).  Notice that the unitaries in $\mathcal M_{n_1}(\mathbb K_1)\oplus\dots\oplus\mathcal M_{n_\ell}(\mathbb K_\ell)$ are of form $U_1\oplus\dots\oplus U_\ell$ where $U_i$ is a unitary matrix in $\mathcal M_{n_i}(\mathbb K_i)$. Therefore,   $A=A_1\oplus\dots\oplus A_\ell$ is a multiple of unitary  if and only if $M_0^\ast(A_i)=\KK_i^{n_i}$ for each $i$.  Also, by Lemma~\ref{lengthmaximal},  there exists non-zero $B$ such that $A^\bot\subsetneq B^\bot$ if and only if  $\mathrm{dim}_{\mathbb K_i}(M_0^{\ast}(A_i))<\mathrm{dim}_{\mathbb K_i}(M_0^\ast(B_i))$  for some $i$.  The two claims combined give the wanted equivalence.

 (i) $\Longleftrightarrow$ (iii). Assume $A=\bigoplus_{k=1}^\ell A_k$ is right-symmetric. We first show that each  $A_k\in\mathcal M_{n_k}(\mathbb K_k)$ is of the same norm. Suppose on the contrary. Without loss of generality $\|A_1\|<\max\{\|A_2\|,\dots,\|A_\ell\|\}=\|A_i\|=1$ where $i\neq 1$ and  consider $X=X_1\oplus(A_2/2)\oplus\dots\oplus (A_\ell/2)$ with $\|X_1\|=1$ and $X_1$ is BJ orthogonal to $A_1$. Then, we have $X\perp A$. Due to right-symmetry of $A$  this implies  $A\perp X$ so there exists a normalized vector $x \in M_0(A)\subseteq \Span\{M_0(A_2),\dots, M_0(A_\ell)\}=\{(0, x_2,\dots, x_\ell) ;\;\; \|A(0, x_2,\dots, x_\ell)\|=\|A\|\cdot \|(0,x_2,\dots,x_\ell)\|\text{ and }x_k\in\mathbb K_k^{n_k}\}$ such that
  $$0=\langle Ax, Xx\rangle_{\FF}
=\sum\limits_{k=2}^\ell \frac{1}{2}\langle A_k x_k,A_k x_k\rangle_{\mathbb F}=\frac{1}{2}\sum_k\|A_k x_k\|^2=\frac{1}{2}\|Ax\|^2\neq0,$$
which is a contradiction. This implies $\|A_1\|=\max\{\|A_2\|,\dots,\|A_\ell\|\}=\|A\|$. Similarly, we get $\|A_i\|=\|A\|$ for all $1\leq i\leq \ell$. Now, we prove that each $A_i$ is a scalar multiple of a unitary matrix.

Let $1\leq i\leq\ell$ be fixed. Recall that $(X_1\oplus\dots\oplus X_\ell)\mapsto (U_1X_1V_1^\ast,\dots,U_\ell X_\ell V_\ell^\ast)$ is an isometry for unitaries $U_k$ and $V_k$ in $\mathcal M_{n_k}(\mathbb K_k)$, and hence induces a BJ isomorphism, so we can assume with no loss of generality that $$A_i=\Sigma_i=\diag(\sigma_1^i,\dots,\sigma_{n_i}^i)$$ with $\sigma_1^i=\dots=\sigma_j^i>\sigma_{j+1}^i\ge\dots\ge \sigma_{n_i}^i\ge0$ for some $j\in\{1,\dots,n_i\}$. We claim that $j=n_i$ i.e. $A_i$ is a scalar matrix $\|A_i\|=\|A\|$, and this  will imply that $A$ is a scalar multiple of unitary.

Assume, if possible, $j\le n_i-1$ and denote $\sigma:=\frac{\sigma_{j+1}^i}{\sigma_j^i}\in[0,1)$, and consider $B=B_1\oplus\dots \oplus B_{\ell}$ where $B_k=A_k$ if $k\neq i$ and
$$B_i=e_1^i(e_1^i)^*+\dots+e_{j-1}^i(e_{j-1}^i)^*+x_j^i(y_j^i)^*+x_{j+1}^i(y_{j+1}^i)^*,$$  where  $\{e_j^i;\;\;1\le j\leq n_i\}$ denotes the standard basis of $\KK_i^{n_i}$ and
$$x_j^i=\frac{e_j^i-e_{j+1}^i}{\sqrt{2}},\quad y_j^i=\frac{\sigma e_j^i
   +e_{j+1}^i  }{\sqrt{1+\sigma^2}},\quad x_{j+1}^i=\frac{e_j^i+e_{j+1}^i}{\sqrt{2}},\quad y_{j+1}^i=\frac{e_j^i-\sigma e_{j+1}^i}{\sqrt{1+\sigma ^2}}.$$
   Notice that $B$ is already in its singular value decomposition and achieves its norm on $y_j^i$, which is mapped into $By_j^i=x_j^i$ while $Ay_j^i=\frac{\sigma_{j+1}^i}{\sqrt{1+\sigma^2}} (e_j^i+e_{j+1}^i)$  is clearly orthogonal to $x_j^i$ in $\langle \cdot,\cdot\rangle_{\FF}$. Thus, $B\perp A$ by Lemma \ref{M_0(A)}. Because of right-symmetricity of $A$ we then have $A\perp B$, so there exists a normalized $w=w_1\oplus\dots\oplus w_\ell \in M_0(A)=M_0^\ast(A)=\bigoplus_1^\ell M_0(A_k)$
   with $\|Aw\|=\|A\|$ and
   $0=\sum_k\langle A_kw_k,B_kw_k
   \rangle_{\FF}$.
   Due to $B_k=A_k$ whenever $k\neq i$ we see that 
   \begin{align*}0&=\sum_{k\neq i} \|A_kw_k\|^2+\langle A_iw_i,B_iw_i\rangle_{\FF}=\|A\|^2\sum_{k\neq i}\|w_k\|^2+\langle A_iw_i,B_iw_i\rangle_{\FF}\\
   &=\|A\|^2\sum_{k\neq i}\|w_k\|^2+\|A_i\|\langle w_i,B_iw_i\rangle_{\FF}\\
   &=\|A\|^2\sum_{k\neq i}\|w_k\|^2+\|A\|\langle w_i,B_iw_i\rangle_{\FF};
   \end{align*} in the one but last equality we used that the restriction of $A_i$  to the $\KK_i$-subspace  $\KK_i^{j}\oplus0_{n_i-j}=M_0(A_i)=M_0^\ast(A_i) \ni w_i$ is a $\sigma_1^i$-multiple of  identity. Notice also that the compression of $B_i$ to  this subspace  equals $\left(\begin{smallmatrix}
      I_{j-1} & 0 \\
 0 & \frac{1-\sigma}{\sqrt{2} \sqrt{1+\sigma ^2}} \\
   \end{smallmatrix}\right)$ and is positive-definite. 
       This gives $w_k=0$ for every $k$, a contradiction.
      
Conversely, if $A$ is a scalar multiple of a unitary then it achieves its norm on every non-zero vector. Consider an arbitrary $B\perp A$; then $B$ achieves its norm on some vector $y$ with $\langle By,Ay\rangle_{\mathbb F}=0$. Since $A$ also achieves its norm on the same vector we see that $A\perp B$ also holds, so that $A$ is a right-symmetric element.
\end{proof}

We note that Lemma \ref{right} implies that the set of scalar multiples of unitary elements is invariant under any BJ isomorphism between two finite-dimensional $C^*$-algebras. We end this section by proving that same holds for smooth elements of finite-dimensional $C^*$-algebras also. We  call $A=\bigoplus_{k=1}^\ell A_k\in\A$ to be \textit{smooth} if there exists exactly one index $i$ such that $\|A_i\|=\|A\|$ and $\mathrm{dim}_{\mathbb K_i}(M_0(A_i))=1$. For example, 
  $A=\big(\bigoplus_{k=1}^{i-1} 0_{n_k}\big)\oplus E_{st}^i\oplus \big(\bigoplus_{k=i+1}^\ell 0_{n_k}\big)$ 
are smooth elements for all matrix units $E_{st}^i\in\mathcal M_{n_i}(\mathbb K_i)$, which can be easily seen by writing them as $E_{st}^i=e_s^i(e_t^i)^\ast$ and using the fact that $\|E_{st}^ix\|=|(e_t^i)^\ast x|=|x_t^i|\le \|x\|$ for $x=\sum_{j=1}^{n_i} x_j^i e_j^i\in\KK_i^{n_i}$, with inequality being strict except if $x=x_t^i e_t^i$.

There is a well known notion of smoothness in general Banach space $V$ that states that a vector $v\in V$ is smooth  if and only if there exists a unique normalized functional $f$ on $V$ such that $f(v)=\|v\|$ (such $f$ is called a supporting functional for $v$). We  prove below that the two  definitions are equivalent. However before we do that let us  note that our definition of smoothness on finite-dimensional $C^*$-algebras is a special case of Holub's condition, see \cite{hennefeld}. The equivalence of Holub's condition and smoothness has been studied by many authors, for a brief survey see \cite{grover}. In particular, it is known that the two definitions are equivalent for finite-dimensional simple or abelian  $C^*$-algebras, see \cite[Theorem 2.1]{hennefeld}, \cite[Theorem 3.1]{abatzoglou}, \cite[Theorem 3.3]{holub}, and  \cite[Corollary 3.3]{keckic2}, \cite[Corollary 2.2]{keckic1}. We show below that the same holds for any finite-dimensional $C^*$-algebra.
\begin{lemma}\label{smooth_one_block}  Let $\A$ be a finite-dimensional $C^*$-algebra over $\mathbb F$. Then, $A\in\A$ is a smooth element if and only if there exists a unique normalized functional $f$ on $\mathcal A$ such that $f(A)=\|A\|$. Furthermore, for a smooth element $A\in\A$, there exists a unique $i$ such that if $x=\bigoplus_{k=1}^\ell x_k\in M_0(A)$, then $x_k=0$ for all $k\neq i$ and we have
 $$A^\bot=\big(0_{n_1}\oplus\dots\oplus 0_{n_{i-1}}\oplus (A_ix_i)x_i^*\oplus 0_{n_{i+1}}\oplus\dots\oplus 0_{n_\ell}\big)^\bot.$$
\end{lemma}
\begin{proof} First, let $A$ be a smooth element of $\A$. By definition  there exists exactly one $i$ such that $\|A_i\|=\|A\|$ which, by Lemma~\ref{M_0(A)}, gives the statement about the vectors in $ M_0(A)$. Moreover,
 we also have $\mathrm{dim}_{\KK_i}(M_0(A_i))=1$, so there exists a  vector  $x_i\in\mathbb K_i^{n_i}$ such that \begin{equation}\label{eqnewfor}
 M_0(A)=0\oplus\dots\oplus 0\oplus x_i\mathbb K_i\oplus 0\oplus\dots\oplus 0
 \end{equation} Then, (i) of Lemma \ref{M_0(A)} and identity \eqref{eq:inner-product-equality}, by which  $\langle A(x_i\gamma),B(x_i\gamma)\rangle_{\FF}=|\gamma|^2\langle Ax_i,Bx_i\rangle_{\FF}$  for each normalized vector $x_i\gamma\in x_i\KK_i= M_0(A_i)$ imply
 $$A^\bot=\big(0_{n_1}\oplus\dots\oplus 0_{n_{i-1}}\oplus (A_ix_i)x_i^*\oplus 0_{n_{i+1}}\oplus\dots\oplus 0_{n_\ell}\big)^\bot.$$ This proves the last statement. 

To prove the first one, let $A$ be a smooth element which attains its norm on $i$-th component and define an $\FF$-linear functional $f$ on $\A$ as $f\colon \big(\bigoplus_{k=1}^\ell X_k\big)\mapsto \frac{1}{\|A_{i}\|}\langle A_ix_i, X_ix_i\rangle_{\mathbb F}$. Then, by Cauchy-Schwarz, $f$ is a normalized functional and  $f(A)=\|A_i\|=\|A\|$; also, $A^\bot=\mathop{\mathrm{Ker}}f$. By~\eqref{outgoingnbd}, the kernel of every supporting functional of $A$ is contained in $A^\bot=\mathop{\mathrm{Ker}}f$ so $f$ is a unique supporting functional of $A$. 

 Conversely, let $A\in\A$ be such that there exists a unique normalized supporting  functional for $A$. Assume there exist $j\neq i$ such that $\|A_i\|=\|A_j\|=\|A\|$. Then there would be normalized  vectors $x\in\mathbb K_i^{n_i}$ and $y\in\mathbb K_j^{n_j}$ such that $\|A_ix\|=\|A\|$ and $\|A_jy\|=\|A\|$ so   $f_x\colon \big(\bigoplus_{k=1}^\ell X_k\big)\to \frac{1}{\|A_i\|}\langle A_ix, X_i\rangle_\FF$ and $f_y\colon \big(\bigoplus_{k=1}^\ell X_k\big)\to \frac{1}{\|A_j\|}\langle A_jy, X_jy\rangle_\FF$ would be two  distinct supporting (normalized) functionals for $A$, which contradicts our assumption that $A$ has a unique supporting (normalized) functional for $A$. So, there exists a unique $i$ such that $\|A_i\|=\|A\|$. Similar arguments prove that  there does not exist two $\KK_i$-linearly independent vectors $x, y\in M_0(A_i)\subseteq\KK_i^{n_i}$. Thus, $\mathrm{dim}_{\KK_i}(M_0(A))=1$.
\end{proof}

The next lemma will show  that smooth elements are preserved under BJ isomorphism.
\begin{lemma}\label{smoothchar} Let $\A$ be a finite-dimensional $C^*$-algebra over $\mathbb F$. Then $A\in\A$ is a smooth element if and only if there does not exist  $B\in\A$ such that $B^\bot\subsetneq A^\bot$.
\end{lemma}
\begin{proof} We assume the usual matrix decompositions \eqref{eq:C*(complex)}--\eqref{eq:C*(real)} of $\A$. Let $A\in\A$ be smooth, achieving its norm only on $i$-th component, and let $B=\bigoplus_{k=1}^\ell B_k$ satisfy $B^\bot\subsetneq A^\bot$. Notice that $M_0^*(B_k)$ is a $\mathbb K_k$-subspace of $\KK_k^{n_k}$ for each $k$. By Lemma \ref{lengthmaximal}, $B^\bot\subsetneq A^\bot$ implies  $\mathrm{dim}_{\KK_k}(M_0^*(B_k))\leq \mathrm{dim}_{\KK_k}(M_0^*(A_k))$. Thus, for all $k\neq i$ we have $\mathrm{dim}_{\KK_k}(M_0^*(B_k))=0$, i.e., $B_k=0$, while  for $k=i$  we have $\mathrm{dim}_{\KK_i}(M_0^*(B_i))\leq 1$ with $M_0^\ast(B_i)\subseteq M_0^\ast(A_i)$. So, either $B_i=0$ or $M_0^\ast(B_i)= M_0^\ast(A_i)$ and $B_i|_{M_0^\ast(A_i)}=\alpha A_i|_{M_0^\ast(A_i)}$ for some non-zero $\alpha\in\FF$. The former case gives $B^\bot=0^\bot=\A$, while  the later case gives $B^\bot= A^\bot$ by  (i) of  Lemma~\ref{M_0(A)}. So $B^\bot\subsetneq A^\bot$ is not possible.

Conversely, assume there does not exist $B\in\A$ such that $B^\bot\subsetneq A^\bot$ and   suppose, if possible, that $A$ is not smooth. Then either there exists two distinct $i, j$ such that $\|A_i\|=\|A_j\|=\|A\|$ or there exists $j$ with $\|A_j\|=\|A\|$ but $\mathrm{dim}_{\mathbb K_j}(M_0(A_j))
>1$. In the first case, by Lemma \ref{M_0(A)}(i),  
 $$\bigg(\big(\bigoplus_{k=1}^i 0_{n_k}\big)\oplus A_i\oplus\big(\bigoplus_{k=i+1}^\ell 0_{n_k}\big)\bigg)^\bot\neq \bigg(\big(\bigoplus_{k=1}^j 0_{n_k}\big)\oplus A_j\oplus\big(\bigoplus_{k=j+1}^\ell 0_{n_k}\big)\bigg)^\bot$$
are  both  properly contained in $A^\bot$ because only the first contains $\big(\bigoplus_{k=1}^j 0_{n_k}\big)\oplus A_j\oplus\big(\bigoplus_{k=j+1}^\ell 0_{n_k}\big)$.
In the second case, let  $x, y\in M_0(A_j)$ be  $\mathbb K_j$-linearly independent. By applying Gram-Schmidt we can assume that their $\KK_j$-valued inner product, $\langle x,y\rangle:=y^\ast x=0\in\KK_j$. Then, by Lemma \ref{M_0(A)}(i), $$\bigg(\big(\bigoplus_{k=1}^i 0_{n_k}\big)\oplus (Ax)x^*\oplus\big(\bigoplus_{k=i+1}^\ell 0_{n_k}\big)\bigg)^\bot\neq \bigg(\big(\bigoplus_{k=1}^i 0_{n_k}\big)\oplus (Ay)y^*\oplus\big(\bigoplus_{k=i+1}^\ell 0_{n_k}\big)\bigg)^\bot$$  are both properly contained in $A^\bot$ because only the first one  contains $\big(\bigoplus_{k=1}^i 0_{n_k}\big)\oplus (Ay)y^*\oplus\big(\bigoplus_{k=i+1}^\ell 0_{n_k}\big)$. 
Either case is contradictory. So $A$ is a smooth element.
\end{proof}

 \section{BJ orthogonality in pseudo-abelian $C^\ast$-algebra}\label{section3}

In the next lemma  we show that BJ orthogonality characterizes the underlying field in case of finite-dimensional abelian $C^*$-algebra $\A=\bigoplus_{k=1}^\ell \mathcal M_1(\mathbb F)=\mathbb F^\ell$ over the field $\mathbb F$. As usual, we prefer to  write its elements as sequences, though we might still use $\oplus$ notation. We will use the notation $$\mathcal R_\A=\{A;\;\; A\text{ is right-symmetric element in } \A\}.$$

\begin{lemma}\label{maximal-Rn}
    Let $\A=\mathbb F^\ell$ be a finite-dimensional abelian  $C^*$-algebra over the field $\mathbb F$ with $\ell\geq 2$. Then, the set $$\{A^\bot;\;\; A\in \mathcal R_\A\setminus\{0\}\}$$
   has finitely many elements in case of $\mathbb F=\mathbb R$ and infinitely many in case of $\mathbb F=\mathbb C$. They are indexed by the tuples $(1,\pm1,\dots,\pm1)$ in case of $\RR^\ell$ and are indexed by  $(1,e^{i\theta_2},\dots,e^{i\theta_\ell})$ for $\theta_k\in[0,2\pi)$ in case of $\CC^\ell$.
\end{lemma}
\begin{proof} Any non-zero right symmetric element is a multiple of unitary by Lemma \ref{right}. Also, $A^\bot=(\lambda A)^\bot$ for $\lambda\in\FF\setminus\{0\}$ because BJ orthogonality is homogeneous. So, we have
$$\{A^\bot;\;\; A\in \mathcal R_\A\}=\{(1,\alpha_{2},\dots,\alpha_{\ell})^\bot ;\;\; |\alpha_i|= 1\}\cup\{\mathbb F^\ell=0^\bot\}.$$  As for the fact that $(1,\alpha_2,\dots,\alpha_{\ell})$ with $|\alpha_i|=1$ have different outgoing neighborhoods, notice that if $A=(1,\alpha_2,\dots,\alpha_{\ell})$ and $B=(1,\beta_2,\dots,\beta_{\ell})$ satisfy $A^\bot=B^\bot$, then by Lemma \ref{lengthmaximal}, there exists $\lambda\in\mathbb F$ such that $$(1,\alpha_2,\dots,\alpha_{\ell})=\lambda (1,\beta_2,\dots,\beta_{\ell}),$$ which implies $\lambda=1$ and $A=B$.\end{proof}

We remark that Lemma \ref{maximal-Rn} does not hold if $\ell=1$ because $\mathcal M_1(\mathbb R)$ and $\mathcal M_1(\mathbb C)$ are BJ isomorphic (see \cite[Example 2.2]{simple}). In the next lemma we give a complete characterization of $\FF^\ell$ among the finite-dimensional pseudo-abelian $C^*$-algebras. We also provide a formula for the dimension of a complex abelian $C^*$-algebra which  uses nothing but  BJ orthogonality relation. It is simpler than the one valid in general normed spaces, see \cite[Theorem 1.1 and Remark 1.2]{guterman}. For simplicity we denote $1$-by-$1$ blocks ${\mathcal M}_1(\KK)$ simply as $\KK$.

\begin{lemma}\label{dimvsleft}
    Let $\A=\KK_1\oplus\dots\oplus\KK_\ell$ be a finite-dimensional pseudo-abelian $C^\ast$-algebra over the field $\FF$. If $\A=\FF^\ell$, then
     \begin{equation}\label{dimvslefteq}\dim\A=|\{A^\bot;\;\;A\in\mathcal L_\A\setminus\{0\}\}|.
     \end{equation} However, if $\FF=\RR$ and one of $\KK_i\in\{\CC, \HH\}$, then $\{A^\bot;\;\;A\in\mathcal L_\A\setminus\{0\}\}$ is an infinite set.
\end{lemma}
\begin{proof} By Lemma~\ref{left}, the set of all non-zero left-symmetric elements in $\A$ is given by (recall that $0^n$ denotes $n$ repeated zeros)
 $$\{(0^{k-1},\alpha_k,0^{\ell-k})
 ;\;\; 1\leq k\leq\ell\text{ and }\alpha_k\in\KK_k\setminus\{0\}\}.$$
Now, all non-zero  $\FF$-multiples of an element share the same outgoing neighbourhood. We further note that, by Lemma \ref{lengthmaximal}(i), $i\neq j$ and $\alpha_i\in\mathbb K_i\setminus\{0\},\alpha_j\in\mathbb K_j\setminus\{0\}$ imply $$A^\bot=(0,\dots,0,\alpha_i,0,\dots,0)^\bot\neq (0,\dots,0,\alpha_j,0,\dots,0)^\bot=B^\bot,$$  (because $M_0(A)\neq M_0(B)$). Therefore, if $\A=\RR^\ell$ or $\CC^\ell$, then \eqref{dimvslefteq} holds.

Finally, let $\A=\KK_1\oplus\dots\oplus\KK_\ell$ be $C^*$-algebra over $\RR$ with $\KK_1\in\{\CC, \HH\}$. Then $$A_\lambda=(1+\lambda \boldsymbol{i},0,\dots,0);\qquad (\boldsymbol{i}^2=-1\in\RR)$$ are left-symmetric elements for all $\lambda\in\mathbb R\setminus\{0\}$ by Lemma \ref{left}. However, $A_\lambda^\bot\neq A_\mu^\bot$ for  $\lambda\neq \mu$ because $\big(1-\frac{1}{\lambda} \boldsymbol{i},0,\dots,0\big)\in (1+\mu \boldsymbol{i},0,\dots,0)^\bot$ if and only if $\mu= \lambda$.\end{proof}

As a direct consequence of Lemma \ref{maximal-Rn} with its proof, and Lemma \ref{dimvsleft}, we get the following corollary which characterizes the underlying fields in an pseudo-abelian $C^*$-algebra with the help of  BJ orthogonality.

\begin{corollary}\label{lem:Rn-vs-Cn}
Let $\A$ be pseudo-abelian $C^\ast$-algebra over the field $\FF$ with $\dim_{\FF}\A\geq 2$. Then $\FF=\CC$ if and only if  \begin{equation*}|\{A^\bot;\;\;A\in\mathcal L_\A\}|<\infty\quad \text{ and }\quad |\{A^\bot;\;\;A\in\mathcal R_\A\}|=\infty.\end{equation*} Moreover, $\FF=\RR$ if and only if  either both  sets are infinite or else they are both finite.
\end{corollary}

Since the above corollary characterizes the underlying field and complex pseudo-abelian finite-dimensional $C^*$-algebras are completely determined by their dimension, which is given by formula \eqref{dimvslefteq}, we already got  a complete BJ characterization of  complex finite-dimensional pseudo-abelian  $C^*$-algebras. It remains to focus on real finite-dimensional pseudo-abelian $C^*$-algebras, where we still need to compute the number of  blocks over reals, over complexes, and over quaternions. This will be done by carefully counting the cardinality associated with finite collections of smooth elements $A_1,\dots, A_s\in\A$. By convention, if $s=0$  we let $\bigcap_{k=1}^sA_k^\bot:=\A$.

\begin{lemma}\label{lem:CCandHH}
Let $\A\in\{\mathcal M_1(\CC), \mathcal M_1(\HH)\}$ be a real $C^\ast$-algebra. Then, there exist finitely many smooth elements $A_1,\dots, A_s$ in $\mathcal A$ such that $$|\{B^\bot;\;\; B\in\mathcal L_{\bigcap_{k=1}^sA_k^\bot\setminus\{0\}}\}|<\infty.$$ If $s$ is the minimal such number, then $s=1$ in case of $\A=\mathcal M_1(\CC)$ and $s=3$ in the case of $\A=\mathcal M_1(\HH)$.  In both cases, $\bigcap_{k=1}^sA_k^\bot$ is a one-dimensional real vector space.
\end{lemma}
\begin{proof}
As usual, $\KK$ will denote either the field $\CC$ or the (skew)  field $\HH$. Recall from Lemma \ref{M_0(A)} that, for $A=(a),B=(b)\in{\mathcal M}_1(\KK)$ ($1$-by-$1$ matrices), we have  $A\perp B$ if and only if $\mathrm{Re}\langle A\cdot 1, B\cdot 1\rangle=\mathrm{Re}(\overline{b} a)=0$ (here, $B\cdot 1$ is matrix $B$ applied on a vector $1\in\KK$).
Therefore, $A^\bot$ coincides with the kernel of the $\RR$-linear functional  $f_A\colon{\mathcal M}_1(\KK)\to\RR$, given by $f_A\colon X=(x)\mapsto  \mathrm{Re}(\overline{x}a)=\mathrm{Re}(\overline{a}x)$. Notice also that the map $W\mapsto f_W$  from $\mathcal M_1(\KK)$ to $\mathrm{Hom}_{\RR}(\mathcal M_1(\KK),\RR)$ is   $\RR$-linear  with zero kernel, because for $W=(w)$, we have  $f_W(W)=\mathrm{Re}(\bar{w}w)=|w|^2=0$ if and only if $W=0$. It implies that $A^\bot=\ker f_A$ and $B^\bot=\ker f_B$ are different whenever $A,B$ are $\RR$-linearly independent. Therefore, if ${\mathcal V}\subseteq\mathcal M_1(\KK)$ is a real subspace of dimension at least two, and $A=(a),B=(b)\in{\mathcal V}$ are $\RR$-linearly independent, then $(A+\lambda_1 B)$ and $(A+\lambda_2 B)$ are $\RR$-linearly independent for $\lambda_1\neq\lambda_2$ and as such $(A+\lambda_1 B)^\bot\neq (A+\lambda_2 B)^\bot$. Thus,  the cardinality of $\{A^
\bot;\;\; A\in{\mathcal V}\}$ is infinite. This shows that $s\ge 1$ in case $\KK=\CC$ and $s\ge 3$ in case of $\KK=\HH$. The inequalities are achieved, for example by using $\{\boldsymbol{i}\}$ in case of $\CC$ and $
\{\boldsymbol{i,j,k}\}$ in case of $\KK=\HH$.
\end{proof}

\begin{lemma}\label{lem:totalnrblocks}
    Let $\A=\KK_1\oplus\dots\oplus\KK_\ell$ be a real finite-dimensional pseudo-abelian $C^*$-algebra. Then, there exists finitely many smooth elements $A_1,\dots, A_s$ in $\mathcal A$ such that \begin{equation}\label{090909}|\{B^\bot;\;\; B\in\mathcal L_{\bigcap_{k=1}^sA_k^\bot}\setminus\{0\}\}|<\infty.\end{equation} If $s$ is minimal such number, then $\bigcap_{k=1}^s A_k^\bot=(\mathbb R\alpha_1,\dots, \mathbb R\alpha_\ell)$
         for some unimodular numbers $\alpha_k\in\mathbb K_k$ and the cardinality of the set in \eqref{090909} is equal to $\ell$, the number of matrix blocks in $\mathcal A$. Furthermore, \begin{equation}\label{dimformula}\dim\A=s+\ell.\end{equation}
\end{lemma}
\begin{proof}  In case of a real $C^*$-algebra $\A=\mathbb R^\ell$  we have, by Lemma \ref{dimvsleft}, $s=0$; clearly also $\dim\mathcal A=\ell$, which, again by  Lemma \ref{dimvsleft}, equals the cardinality of \eqref{090909},    and  the statement  follows by inserting $\alpha_k=1$. We now consider the remaining cases of a real $C^*$-algebra when one of the blocks is $\mathbb C$ or $\mathbb H$ with $\mathbb F=\mathbb R$. Without loss of generality, let $\A=\RR^r\oplus\CC^c\oplus\HH^h$ for some $r,c,h\ge0$. Now, if we take $A_k$'s to be all the elements in the finite set $\{(0^j,\mu,0^{\ell-j-1})
 ;\;\; r+1\leq j\leq \ell, \ \mu\in\{\boldsymbol{i,j,k}\}\cap\mathbb K_j\}$, then $A_k$ are smooth elements. It is straightforward that $(\alpha_1,\dots,\alpha_\ell)
 \in\bigcap_{k=1}^s A_k^\bot$ is possible only if all $\alpha_k\in\RR$, so $$\bigcap_{k=1}^s A_k^\bot=\RR^r\oplus(\mathbb R+0\boldsymbol{i})^c\oplus(\mathbb R+0\boldsymbol{i}+0\boldsymbol{j}+0\boldsymbol{k})^h.$$ By Lemma \ref{left}, all left-symmetric elements in $\bigcap_{k=1}^s A_k^\bot\setminus\{0\}$ are of the form 
 $(0^{j-1},\alpha,0^{\ell-j})$
 and $\alpha\in\RR\setminus\{0\}$. 
Note that $(0^{j-1},\alpha,0^{\ell-j})^\bot$
 equals $(0^{j-1},1,0^{\ell-j})^\bot$
 So, $$\{B^\bot;\;\; B\in\mathcal L_{\bigcap_{k=1}^sA_k^\bot}\setminus\{0\}\}=
\{B^\bot;\;\; B=(0^{j-1},1,0^{\ell-j})\text{ for }1\leq j\leq \ell\}
  ,$$ which is a finite set, and \eqref{090909} holds for some finite $s$.

Let $s$ be minimal such  and let $A_1,\dots, A_s\in\A$ be the corresponding smooth elements for which \eqref{090909} holds. Being smooth, Lemma~\ref{smooth_one_block} implies that  
$$A_i^\bot=(0^{j-1},\mu_{A_i},0^{\ell-j})^\bot$$ for some $j$ (depending on $i$) and a unimodular $\mu_{A_i}\in\KK_j$. Hence, $A_i^\bot$  is an $\RR$-vector subspace of $\A=\KK_1\oplus\dots\oplus\KK_\ell$ having only $j$-th block  different from $\KK_j$.  It implies that $\bigcap_{k=1}^s A_k^\bot$ is also an $\RR$-vector subspace of $\A$. Moreover,  by Lemma \ref{lem:CCandHH},  its  $j$-th block is at most one-dimensional real vector space, else $\bigcap_{k=1}^\ell A_k$ contains infinitely many left-symmetric elements with pairwise distinct outgoing neighborhoods (relative to $j$-th block $\mathcal M_1(\mathbb K_j)=\KK_j$, hence also relative to $\A$), which would  contradict the choice of $s$. 
Thus, $\bigcap_{k=1}^s A_k^\bot=\bigoplus_{k=1}^\ell \mathbb R\alpha_k$ for some  numbers $\alpha_k\in\mathbb K_k$ which we can assume to be either unimodular~or~$0$.

To finish, define a map $\sgno\colon\KK\to\KK$ by  $\sgno(\alpha)=1$ if $\alpha=0$, else $\sgno(\alpha)=\frac{\overline{\alpha}}{|\alpha|}$ and observe that  $x\rightarrow (\sgno(\alpha_1)\oplus\dots\oplus \sgno(\alpha_\ell))x$ is an isometry of $\A$, so induces a BJ isomorphism. With its help we can achieve that $\bigcap_{k=1}^s A_k^\bot=\bigoplus_{k=1}^\ell \RR \alpha_k$ with each $\alpha_k\in\{0,1\}$. This, in turn,  is BJ isomorphic to $\mathbb R^m$, where $m\leq \ell$ is the number of non-zero $\alpha_i$. But since $\mathbb R^\ell$ also contains finitely many left-symmetric elements with pairwise distinct outgoing neighborhoods then, by the minimality of $s$, we have $m=\ell$ and so $\alpha_k=1$ for each $k$. Thus, $\bigcap_{k=1}^s A_k^\bot=\bigoplus_{k=1}^\ell \mathbb R\alpha_k$ for some unimodular numbers $\alpha_k\in\mathbb K_k$ and it contains exactly $\ell$ non-zero left-symmetric elements with pairwise distinct outgoing neighborhoods.  This shows that the set in \eqref{090909} has cardinality  $\ell$. By Lemma \ref{lem:CCandHH}, we furthermore have \begin{equation}\label{eq:s=c+3h}s=c+ 3h.\end{equation} Now, the dimension of $\A$ is clearly equal to $$\dim_{\RR}\mathcal A=r+2c+4h.$$
while the number of blocks satisfies $\ell=r+c+h$. This implies $s=\dim\mathcal A-\ell.$
\end{proof}

We say that a subset $\mathcal S\subseteq\mathcal A$ has property $\mathcal{FL}$ if $|\{A^\bot; A\in\mathcal L_{\mathcal S}\}|<\infty.$

\begin{lemma}\label{lem:nr-real-block} Let $\A$ be a finite-dimensional pseudo-abelian $C^*$-algebra over $\mathbb R$ with $\dim\A\geq 2$. Choose the minimal integer $s$ and the corresponding smooth elements $A_1,\dots, A_s$ as in Lemma~\ref{lem:totalnrblocks}, and define the set
$$\Xi:=\{ A\in \mathcal L_{\A};\;\; \exists \,m \hbox{ such that } A^\bot\cap\bigcap_{k\neq m} A_k^\bot \text{ has property }\mathcal{FL}\}.$$
Then,
$$\Omega=\bigcap_{A\in\Xi}A^\perp$$
consists exactly of elements in $\A$ that are  zero in nonreal blocks of $\A$. Hence, the cardinality of
$$\{A^\perp;\;\;A\in\Omega\cap\mathcal L_{\A}\setminus\{0\} \}$$
coincides with the number of real $1$-by-$1$  blocks in $\A$.
\end{lemma}
\begin{proof} Let $\A=\mathbb R^r\oplus\CC^c\oplus\mathbb H^h$.
Using Lemma \ref{lem:totalnrblocks}, $$\bigcap_{k=1}^s A_k^\bot=\bigoplus_{k=1}^\ell\RR\alpha_k$$ for some unimodular numbers $\alpha_k\in\mathbb K_k$. Now, without loss of generality (by multiplying with a suitable unitary element,  i.e., applying a suitable  isometry), $\alpha_k=1$ for each $1\leq k\leq \ell$. Since $A_k$ are smooth there exist $X_k$ such that $A_k^\bot = X_k^\bot$, where collection of all $X_k$ takes the form $(0^{q-1},\boldsymbol{i},0^{\ell-q})$ in complex blocks, i.e., for $r+1\leq q\leq r+c$, and the form $(0^{q-1},\mu_{j,q},0^{\ell-q})$
 with $\mathrm{span}\{\mu_{1,q},\mu_{2,q},\mu_{3,q}\}=\mathrm{span}\{\boldsymbol{i,j,k}\}$ ($1\leq j\leq 3$) in quaternionic blocks, i.e., for $r+c+1\leq q\leq \ell$ (because we assumed $\alpha_k=1$). Without loss of generality, we assume $A_k=X_k$.

Consider $q$ such that $\mathbb K_q\in\{\CC,\mathbb H\}$. We examine only the slightly more challenging case of $\KK_q=\HH$ in the sequel. Then there exists $m$ such that $A_m=(0^{q-1},\mu_{1,q},0^{\ell-q})$.
 Let us  replace $A_m$ with $$A=(0^{q-1},1,0^{\ell-q}).$$ Now, for each unimodular $\mu\in\mathbb K_q$ we have $\mu^\bot=\mathop{\mathrm{ker}} f_\mu$ (here, $\mu^\bot$ denotes the relative outgoing neighborhood inside the $q$-th block ${\mathcal M}_1(\KK_q)=\KK_q$  and $f_\mu(x)=\mathrm{Re}(\overline{\mu}x)$ is $\mathbb R$-linear functional on $\mathbb K_q$) is a three-dimensional subspace of $\mathbb K_q$. Then, if $\mu_1,\mu_2,\mu_3$ are $\RR$-linearly independent, then $f_{\mu_1},f_{\mu_2},f_{\mu_3}$ are also $\RR$-linearly independent. Now, since $\mu_{1,q},\mu_{2,q}, \mu_{3,q}$ are purely imaginary and $\RR$-linearly independent, it implies that $1, \mu_{2,q}, \mu_{3,q}$ are $\RR$-linearly independent, so $1^\bot\cap \mu_{2,q}^\bot\cap\mu_{3,q}^\bot$ is a one-dimensional subspace in ${\mathcal M}_1(\KK_q)$ (again, we are using here the relative outgoing neighborhoods, within $q$-th block ${\mathcal M}_1(\KK_q)$ only). So, $A^\bot\cap\bigcap_{k\neq m}  A_k^\bot=\bigoplus_{k\neq q}\mathbb R\alpha_k\oplus \mathbb R\beta$ for some unimodular $\beta\in\mathbb K_q$ (because by replacing $A_m^\bot$ with $A^{\bot}$ affects the $q$-th block only), which contains only finitely many left-symmetric elements with pairwise distinct outgoing neighborhoods. It implies
$$(0^{q-1},1,0^{\ell-j})\,,\, (0^{q-1},\mu_{q,j},0^{\ell-j})\in\Xi\quad\text{ for }j\in\{1,2,3\}.$$
 Since $\mathrm{span}(1, \mu_{1,q},\mu_{2,q},\mu_{3,q})=\mathbb K_q$, then $\Omega=\bigcap_{A\in\Xi}A^\perp$ can only contain  elements with zero entries in $q$-th block. Since $\mathbb K_q$ was arbitrary non-real block, then $\Omega$ can only contain elements with zero entry in all non-real blocks of $\A$, that is,
$$\Omega\subseteq\RR^r\oplus 0^c\oplus 0^h.$$

Now we consider those $i$ such that $\mathbb K_i=\mathbb R$. By the minimality of $s$ and the fact that we could assume $A_k=X_k$, each $A_k$ has zero entries in real blocks. Therefore, if we replace some $A_m$ (which has a non-zero entry only in the $q$-th block) with a left-symmetric $A$, as outlined by the procedure, then  we claim  that $A$ also will have a non-zero entry in the $q$-th block: In case this block is complex, then by the minimality of $s$, $A_m$ is the only element among $\{A_1,\dots, A_s\}$  with non-zero entry in $q$-th block. Now, if its substitute, $A\in\Xi$ would have a zero entry in $q$-th block, then 
$(0^{q-1}\oplus {\mathcal M}_1(\CC)\oplus 0^{\ell-q})\subseteq A^\bot\cap \bigcap_{k\neq m} A_k^\bot$. 
 This, by Lemma~\ref{lem:CCandHH}, contradicts the fact that $ A^\bot\cap \bigcap_{k\neq m} A_k^\bot$ has property $\mathcal{FL}$.

 The arguments when the $q$-th block is quaternionic are similar; the only difference is that by the minimality of $s$ (and Lemma~\ref{lem:CCandHH}) we now have three matrices among $\{A_1,\dots,A_s\}$ with  non-zero entries in this block, and they must be $\RR$-linearly independent. Then the substitute, $A\in\Xi$ must again have non-zero entries only in $q$-th block.

Therefore. $\RR^r\oplus 0^c\oplus 0^h\subseteq A^\bot\cap \bigcap_{k\neq m} A_k^\bot$ for every $A\in\Xi$ and so $\RR^r\oplus 0^c\oplus 0^h\subseteq\Omega\subseteq \RR^r\oplus 0^c\oplus 0^h$, as claimed. The claim about the number of non-zero left symmetric elements with  property $\mathcal{FL}$ inside $\Omega$ is now clear.
\end{proof}

\section{Proofs of Main results}\label{sect:prfs}
\begin{proof}[Proof of Theorem \ref{leftsymmetric}]
 Without loss of generality, $\A=\bigoplus_{k=1}^\ell \mathcal M_{n_k}(\mathbb K_k)$ with $n_1=\dots=n_p=1$ and $n_{p+1},\dots, n_\ell\neq 1$ for some $p\geq 0$. If $p=0$, then, by Lemma~\ref{left}, $\mathcal L_\A=\{0\}$, so $\mathcal L_\A^\bot=\A$. This  matches with the sum of  minimal  ideals which are not skew-fields.  Also, $\mathcal L_\A^{\bot\bot}=\{0\}$, which agrees with the statement when there are no skew-field minimal ideals.

If $p\neq 0$, then, by  Lemma \ref{left},
$$\mathcal L_\A=\bigcup_{1\leq i\leq p}\big(\bigoplus_{k=1}^{i-1}0_{n_k}\big)\oplus \mathcal M_1(\KK_i)\oplus \big(\bigoplus_{k=i+1}^\ell 0_{n_k}\big).$$
Note that $\mathcal M_1(\KK_i)^\bot=\{0_{n_i}\}$ because $A_{i}\not\perp A_i$ for any non-zero $A_i\in\mathcal M_1(\KK_i)$. Thus, if $A=\alpha_1\oplus \dots\oplus \alpha_p \oplus A_{p+1}\oplus \dots\oplus A_\ell\in {\mathcal L_\A^\bot}$, then the left-symmetric element $\big(\bigoplus_{k=1}^{j-1}0_{n_k}\big)\oplus \alpha_j\oplus \big(\bigoplus_{k=j+1}^{\ell}0_{n_k}\big)$ 
is orthogonal to $A$; giving that $\alpha_j=0$ for all $1\leq j\leq p$. Hence,
$$\mathcal L_\A^\bot:=\bigcap_{A\in\mathcal L_\A}A^\bot= \big(\bigoplus_{k=1}^p0_{n_k}\big)\oplus\bigoplus_{k=p+1}^\ell \mathcal M_{n_k}(\KK_k),$$
which coincides with the sum of  minimal ideals that are not skew-fields.  Moreover,
 $$\mathcal L_\A^{\bot\bot}:=\bigcap_{A\in \mathcal L_\A^\bot} A^\bot=\bigoplus_{k=1}^p\mathcal M_1(\KK_k)\oplus\big(\bigoplus_{k=p+1}^\ell0_{n_k}\big),$$ which is the sum of skew-field minimal ideals. \end{proof}

 \begin{proof}[Proof of Theorem \ref{abelian}] Assume there is a BJ isomorphism between $\A$ and $\A'$ and $\A$ is a finite-dimensional pseudo-abelian $C^*$-algebra. If $\phi$ is a BJ isomorphism between $\A$ and $\A'$, we get $\mathcal L_{\A'}=\phi(\mathcal L_\A)$, $\mathcal R_{\A'}=\phi(\mathcal R_\A)$ and $\phi(0)=0$ (since $x=0$ is the only element with $x\perp x$). Using Corollary \ref{abeliancharacterization}, we get $\A'$ is pseudo-abelian. Using BJ isomorphism of $\A$ and $\A'$, we have that the cardinalities of $\{A^\bot;\;A\in\mathcal B \text{ is left-symmetric}\}$ and $\{A^\bot;\;A\in\mathcal B \text{ is right-symmetric}\}$ are same for $\mathcal B\in\{\A,\A'\}$. Then, using Corollary \ref{lem:Rn-vs-Cn}, we get $\FF=\FF'$. For $\FF=\CC$, then result follows because the dimensions of $\A$ and $\A'$ are same using \eqref{dimvslefteq}. Now, we consider the case $\FF=\FF'=\RR$.

Let $\A=\RR^r\oplus\CC^c\oplus\HH^h$ and $\A'=\RR^{r'}\oplus\CC^{c'}\oplus\HH^{h'}$. Then, by Lemmas~\ref{smoothchar} and \ref{lem:nr-real-block} we have $$r=r'.$$
By Lemma \ref{lem:totalnrblocks}, we also have that the number of blocks in $\A$ and $\A'$ are same, i.e., $$r+c+h=r'+c'+h'.$$ The minimal number $s$ such that there exists smooth elements $A_1,\dots, A_s$ for which \eqref{090909} holds is also preserved under a BJ isomorphism. Hence, by \eqref{dimformula},  the dimensions of $\A$ and $\A'$ are same,~so $$r+2c+4h=r'+2c'+4h'.$$ It implies, $r=r', c=c'$ and $h=h'$. Consequently, $\A$ and $\A'$ are $C^\ast$-isomorphic.
\end{proof}
 \section{Extraction of abelian summand}\label{section4}

Recall from Theorem \ref{leftsymmetric} that, given a finite-dimensional $C^*$-algebra $\A$, 
\begin{align*}
   \mathcal L_{\A}^{\bot\bot}&=\mathbb C^c;\quad\quad\hbox{ if } \FF=\CC, \\ 
   \intertext{or}
   \mathcal L_{\A}^{\bot\bot}&=\mathbb R^r\oplus\mathbb C^c\oplus\mathbb H^h;\quad \hbox{ if } \mathbb F=\mathbb R,
\end{align*}
 where $r,c,h$ are the numbers of $1$-by-$1$ real, complex and quaternionic blocks in the matrix block decomposition, respectively. Notice that in case of complex $C^*$-algebra, its pseudo-abelian and abelian summands coincide and equal to $\mathbb C^c$. However, for a real $C^*$- algebra its abelian summand equals $\RR^r\oplus \CC^c$ and differs from the pseudo-abelian summand when  $h>0$. Now, we give a procedure to classify the abelian summand in case of real $C^*$-algebra. Recall that $\FF$, the underlying field of $\A$, can be determined using BJ orthogonality alone when $\dim\mathcal L_{\A}^{\bot\bot} \geq 2$, see Corollary \ref{lem:Rn-vs-Cn} and note that the dimension of the pseudo-abelian $C^*$-algebra $L_{\A}^{\bot\bot}$ can be computed using \eqref{dimformula} (this was proven for real $C^*$-algebra but it holds even for complex $C^*$-algebra since then $s=0$ and  \eqref{dimformula} reduces to \eqref{dimvslefteq}).  
   When $\dim\mathcal L_{\A}^{\bot\bot} = 1$, then the pseudo-abelian and abelian summand of $\A$ coincides and equal to $\mathbb C$ (if $\FF=\CC$) or $\mathbb R$ (if $\FF=\RR$), respectively. 
 
 Clearly, to extract the abelian summand we only need to work within the pseudo-abelian summand $\mathcal L_{\A}^{\bot\bot}$ and we only need to consider real $C^*$-algebras, that is,  $\FF=\RR$. We will require the smooth points in $\mathcal L_{\A}^{\bot\bot}$. Since $\mathcal L_{\A}^{\bot\bot}$ is a $C^*$-algebra, its smooth points can be described by BJ orthogonality alone, see Lemma \ref{smoothchar}. We will also require a property similar to the property $\mathcal{FL}$, which was defined just prior Lemma~\ref{lem:nr-real-block}. We say that a subset $\mathcal S\subseteq \mathcal L_{\A}^{\bot\bot}$ has a relative property $\mathcal{FL}$ with respect to pseudo-abelian summand, $\mathcal{PFL}$ for short, if $$|\{X^\bot\cap \mathcal L_{\A}^{\bot\bot};\;\;X\in \mathcal L_{\mathcal S}\}|<\infty.$$ 

Now, we consider the following procedure to extract quaternionic $1$-by-$1$ blocks in $\mathcal L_{\A}^{\bot\bot}$: Start with an arbitrary finite set $\Omega\subseteq\mathcal{L}_\A^{\bot\bot}$, with property $\mathcal{PFL}$, that consists of smooth points relative to $\mathcal{L}_\A^{\bot\bot}$ and is of minimal possible cardinality (it exists by Lemma \ref{lem:totalnrblocks}). By Lemma~\ref{M_0(A)} we replace every element $S\in\Omega$ by  $\hat{S}\in{\mathcal L}_\A^{\bot\bot}$, which is left symmetric relative to ${\mathcal L}_\A^{\bot\bot}$ and satisfies $S^\bot\cap {\mathcal L}_\A^{\bot\bot}=\hat{S}^\bot\cap {\mathcal L}_\A^{\bot\bot}$. This way we achieve that each element of $\Omega$ is left-symmetric relative to ${\mathcal L}_\A^{\bot\bot}$ and, as such, belongs to a single block of ${\mathcal L}_\A^{\bot\bot}$ (see Lemma~\ref{left}(ii))   It follows from the proof of equation \eqref{eq:s=c+3h} in Lemma \ref{lem:totalnrblocks} that, due to  its minimal cardinality,  $|\Omega|=c+3h$ and no element in $\Omega$  belongs to a real block of $\mathcal{L}_\A^{\bot\bot}$, while  each complex block of $\mathcal{L}_\A^{\bot\bot}$ has one and each quaternionic block of $\mathcal{L}_\A^{\bot\bot}$ has three representatives in $\Omega$. 
 
 Thus, if $|\Omega|\le 2$, then ${\mathcal L}_\A^{\bot\bot}$ is the abelian part (and $\A$ is abelian if and only if $\A={\mathcal L}_\A^{\bot\bot}$). If $|\Omega|\ge 3$, let $\Omega'$ be the collection of all $3$-subsets (i.e., subsets of cardinality $3$) of $\Omega$. For each $\{X_1,X_2,X_3\}\in\Omega'$ we select (if they  exist) all  non-zero $X\in  {\mathcal L}_\A^{\bot\bot}$, left-symmetric relative to ${\mathcal L}_\A^{\bot\bot}$, such that  the three sets
\begin{equation}\label{eq:abalianaxtracton}
\bigcap_{S\in\Omega\setminus\{X_1\}}S^\bot\cap X^
\bot\cap{\mathcal L}_\A^{\bot\bot},\qquad \bigcap_{S\in\Omega\setminus\{X_2\}}S^\bot\cap X^
\bot\cap{\mathcal L}_\A^{\bot\bot},\qquad \bigcap_{S\in\Omega\setminus\{X_3\}}S^\bot\cap X^
\bot\cap{\mathcal L}_\A^{\bot\bot}
\end{equation}
have property $\mathcal{PFL}$.  
 By Lemma~\ref{lem:CCandHH}  every quaternionic $1$-by-$1$ block contains such a representative triple  in $\Omega$ (e.g., if $\{X_1,X_2,X_3\}=(0\oplus \{\boldsymbol{i,j,k}\}\oplus0) \subseteq0\oplus \HH\oplus0$;  we can take $X=(0\oplus 1\oplus 0)\in0\oplus \HH\oplus 0$). Conversely if, for a triple $\{X_1,X_2,X_3\}\subseteq\Omega$, we can find such a left-symmetric $X$, then, by  Lemma \ref{left}(ii), $X$ belongs to a single block of ${\mathcal L}_\A^{\bot\bot}$. It is then easy to see that if  $\{X_1,X_2,X_3\}$ do not belong to the same block (necessarily quternionic), then, by the minimal cardinality of $\Omega$, at least one of the  three sets in \eqref{eq:abalianaxtracton}  will not have the property $\mathcal{PFL}$, a contradiction. 

One also sees that each $X$ must belong to the same quaternionic block containing $X_1,X_2,X_3$ and at least  one, say $X_0$,  is not in their $\RR$-linear span. Then, $X_1^\bot \cap X_2^\bot\cap X_3^\bot\cap X_0^\bot$ vanishes on this quaternionic block.
Therefore, the common outgoing neighborhood  of all those triples, together with all the adjourned vertices $X$, and intersected by ${\mathcal L}_\A^{\bot\bot}$, is the abelian summand of $\A$.

To summarize the complete extraction of abelian summand:  Start with $\mathcal{L}_\A^{\bot\bot}$. If $\dim\mathcal{L}_\A^{\bot\bot}=1$ (c.f.~\eqref{dimformula}) or if $\dim\mathcal{L}_\A^{\bot\bot}\ge 2$ and $\FF=\CC$ (see Corollary \ref{lem:Rn-vs-Cn}), then $\mathcal{L}_\A^{\bot\bot}$ is the abelian summand. Otherwise, apply the above procedure to get it.

\iffalse\begin{theorem} Let $\A$ be a finite-dimensional $C^*$-algebra. Then there exists a finite tuple $A_1,\dots, A_s\in \mathcal L_\A^{\bot\bot}$  such that 
$$\{X^\bot;\;\;X\in A_1^\bot\cap\dots\cap A_s^\bot \cap \mathcal L_\A^{\bot\bot}\}$$ 
is finite. Moreover, with any such tuple of  minimal possible length,  define
$$\Omega\{\{X,Y\}\subseteq \{A_1,\dots, A_s\};\;\; X\neq Y \;\hbox{and}\;\exists B\in\mathcal L_\A^{\bot\bot} \hbox{ with } \bigcap_{A\in\{A_1,\dots,A_s\}\setminus\{X,Y\} } A^\bot \hbox{has property } {\mathcal P}\} $$
Then, $\Omega^\bot$ is the non-abelian summand of $\A$, while $\Omega^{\bot\bot} $ is its abelian summand.
\end{theorem}
\fi
\begin{corollary}
    Let $\A$ be a finite-dimensional $C^*$-algebra over $\FF$. Then, the following are equivalent:
    \begin{enumerate}
        \item[(i)] $\A$ is abelian. 
        \item[(ii)] $\A=\mathcal L_\A^{\bot\bot}$ and it contains no quaterninic blocks.
        \item[(iii)] $\A=\mathcal L_\A^{\bot\bot}$ and if $A_1,\dots, A_s\in\A$ is any (hence every) minimal tuple with the property ${\mathcal FL}$, then it contains no triple for which a  non-zero left symmetric $X$ would exist so that  \eqref{eq:abalianaxtracton} would have property ${\mathcal FL}$.
    \end{enumerate}
\end{corollary}

\bigskip
\noindent
{\bf \Large Acknowledgment}. The authors are indebted to Professor  R. Tanaka for providing them with references \cite{Li03} and \cite{Rosenberg}.

The research of Bojan Kuzma is
supported in part by the Slovenian Research Agency (research program P1-0285 and research projects N1-0210, N1-0296 and J1-50000). The research of Sushil Singla is supported in part by Slovenian Research Agency (research project N1-0210),  and JCB/2021/000041 of SERB, India.


\begin{thebibliography}{WW}

\bibitem{abatzoglou} T.J. Abatzoglou, \textit{Norm derivatives on spaces of operators}, Mathematische Annalen, {\bf 239} (1979), 129--135.

 
\bibitem{blackadar} B. Blackadar, \textit{Operator Algebras-Theory of $C^*$-Algebras and von Neumann Algebras}, Springer, 2006.

\bibitem{bhatia} R. Bhatia, P. \v Semrl, {\it Orthogonality of matrices and some distance problems}, Linear Algebra Appl. {\bf 287} (1999), 77--86.

\bibitem{Con01} C. Constantinescu, $C^*$-algebras. Vol. 1--5, volume 58-62 of
North-Holland Mathematical Library, North-Holland Publishing Co., Amsterdam, 2001.

\bibitem{conway} J.B. Conway, {\it A Course in Functional Analysis}, Second edition, Grad. Texts in Math., 96, Springer, New York, 1990.

\bibitem{goodearl} K.R. Goodearl,
\textit{Notes on real and complex $C^*$-algebras}, Shiva Math. Ser., 5, Shiva Publishing Ltd., Nantwich, 1982.
\bibitem{grove}L.C. Grove, Algebra, Pure Appl. Math., 110, Academic Press, Inc., New York,  1983.

\bibitem{guterman} A. Guterman, B. Kuzma, S. Singla, S. Zhiliana, \textit{Birkhoff--James classification of  norm's properties}, Adv. Oper. Theory {\bf 9} (2024), Article no. 43, a special issue dedicated to to Professor Chi-Kwong Li on the occasion of his 65th birthday.

\bibitem{grover} P. Grover, S. Singla, \textit{Birkhoff-James orthogonality and applications: A survey}, Operator Theory, Functional Analysis and Applications, eds. M. A. Bastos, L. Castro, A. Y. Karlovich, Oper. Theory Adv. Appl., Birkh\"{a}user Cham, {\bf 282} (2021), 293--315.

\bibitem{grover1} P. Grover, S. Singla, \textit{Subdifferential set of the joint numerical radius of a tuple of matrices}, Linear Multilinear Algebra {\bf 71} (2023), no. 17, 2709--2718.

\bibitem{hennefeld} J. Hennefeld, \textit{Smooth, compact operators}, {Proceedings of the American Mathematical Society}, {\bf 77} (1979),  87--90.

\bibitem{holub} J.R. Holub, \textit{On the metric geometry of ideals of operators on Hilbert space}, {Mathematische Annalen}, {\bf 201} (1973), 157--163.

  \bibitem{james1} R.C. James, Orthogonality and linear functionals in normed linear spaces. Trans. Amer. Math. Soc. {\bf 61} (1947), 265--292.

\bibitem{keckic2} D.J. Ke$\check{\text c}$ki$\acute{\text c}$, \textit{Gateaux derivative of $B(H)$ norm}, {Proceedings of the American Mathematical Society}, {\bf 133} (2005), 2061--2067.

\bibitem{keckic1} D.J. Ke$\check{\text c}$ki$\acute{\text c}$, \textit{Orthogonality and smooth points in $C(K)$ and $C_b(\Omega)$}, {Eurasian Mathematical Journal}, {\bf 3} (2012), 44--52.
\bibitem{Komuro-Saito-Tanaka} N. Komuro, K.-S. Saito,  R. Tanaka,  \textit{On symmetry of Birkhoff orthogonality in the positive cones of $C^*$-algebras with applications}, J. Math. Anal. Appl. 474 (2019), no. 2, 1488--1497
\bibitem{simple} B. Kuzma and S. Singla, \textit{Non-linear classification of finite-dimensional simple $C^*$-algebras}, accepted for publication in Filomat.

\bibitem{Li03} B. Li, \textit{Real operator algebras}, World Scientific Publishing Co. Inc., River Edge, NJ, 2003.

\bibitem{Rosenberg} J. Rosenberg, \textit{Structure and applications of real $C^\ast$-algebras}, Operator algebras and their applications, 235--258. Contemp. Math., 671 
American Mathematical Society, Providence, RI, 2016.

\bibitem{Sain-Ghosh-Paul}D. Sain, P. Ghosh, K. Paul, \textit{On symmetry of Birkhoff–James orthogonality of linear operators on finite-dimensional real Banach spaces}, Oper. Matrices 11 (2017) 1087--1095.

\bibitem{Sch93} H. Schr\"oder, \textit{K-theory for real $C^*$-algebras and applications},
volume 290 of Pitman Research Notes in Mathematics Series, Longman Scientific \& Technical, Harlow, 1993.
 
\bibitem{tanaka1} R. Tanaka, {\it Nonlinear equivalence of Banach spaces based on Birkhoff-James orthogonality, II}, J. Math. Anal. Appl. {\bf 514} (2022), no. 1, Paper No. 126307, 19 pp.

\bibitem{tanaka3} R. Tanaka, \textit{A Banach space theoretical characterization of abelian $C^*$-algebras}, Proc. Amer. Math. Soc. Ser. B {\bf 10} (2023), 208--218.

\bibitem{rightsymmetricturnsek} A. Turn\v{s}ek, \textit{On operators preserving James' orthogonality}, Linear Algebra Appl. {\bf 407} (2005), 189--195.

\bibitem{zhang} F. Zhang, \textit{Quaternions and matrices of quaternions}, Linear Algebra Appl. {\bf 251} (1997), 21--57.
\end{thebibliography}
\end{document}